\documentclass{amsart}
\usepackage{amsmath,amssymb,verbatim,amscd,bm,color,ulem,calc,accents}
\usepackage[driverfallback=dvipdfm]{hyperref}

\newtheorem{thm}{Theorem}
\newtheorem{cor}{Corollary}
\newtheorem{lem}{Lemma}[section] 
\newtheorem{prop}[lem]{Proposition}

\theoremstyle{definition} 
\newtheorem{defn}{Definition}[section]
 
\newtheorem{rem}{Remark}[section] 
 
\newtheorem*{ack}{Acknowledgement} 
\numberwithin{equation}{section}

 \newcommand{\R}{\mathbb R} 
 \newcommand{\C}{\mathbb C} 
 \newcommand{\T}{\mathbb T} 
 \newcommand{\N}{\mathbb N}

 \newcommand{\la}{\lambda} 
 \newcommand{\D}{\displaystyle} 
 \newcommand{\A}{\mathcal{A}}

 \newcommand{\q}{\quad} 
 \newcommand{\qq}{\qquad} 
  
 \newcommand{\ph}[1]{\phantom{#1}} 
 \newcommand{\Ch}[1]{\operatorname{Ch}(#1)}

 \newcommand{\ov}{\overline}

 \newcommand{\fr}{\frac} 
  
 \newcommand{\cl}[1]{\operatorname{cl}(#1)}
 \newcommand{\clk}[2]{\operatorname{cl}_{#1}(#2)}
 \newcommand{\set}[1]{\{ #1 \}}

 \newcommand{\cha}{\Ch{A}} 
 \newcommand{\chb}{\Ch{\BT}}

 \newcommand{\qbox}[1]{\q \mbox{#1}}

 \newcommand{\al}{\alpha} 
  
 \newcommand{\gam}{\gamma}
 \renewcommand{\Re}{{\rm Re}\,} 
 \renewcommand{\Im}{{\rm Im}\,} 
  
 \renewcommand{\d}{\delta} 
 \renewcommand{\set}[1]{\{ #1 \}} 
 \renewcommand{\ph}[1]{\phantom{#1}}

 \newcommand{\V}[1]{\Vert #1 \Vert} 
  
 \newcommand{\DV}[1]{\left\Vert #1 \right\Vert} 
 \newcommand{\VC}[1]{\left\Vert #1 \right\Vert_{(C)}} 
 \newcommand{\VS}[1]{\left\Vert #1 \right\Vert_{(\Sigma)}}
 \newcommand{\Vs}[1]{\left\Vert #1 \right\Vert_{(\sigma)}}
 \newcommand{\VSS}[1]{\left\Vert #1 \right\Vert_{(e)}}
 \newcommand{\Vinf}[1]{\Vert #1 \Vert_\infty}

 \newcommand{\tf}{\widetilde{f}}
 \newcommand{\dphi}[1]{\d_{\Phi(#1)}}
 \newcommand{\e}{\varepsilon} 
 \newcommand{\ez}{s_0} 
 \newcommand{\eo}{s_1}
 \newcommand{\et}{s_2}
 \newcommand{\id}{{\rm id}}
 \newcommand{\alo}{\al_1}
 \newcommand{\phio}{\phi_1}
 \newcommand{\phii}{\phi_i}
 \newcommand{\phila}{\phi_\la}
 \newcommand{\varphio}{\varphi_1}
 \newcommand{\varphii}{\varphi_i}
 \newcommand{\varphila}{\varphi_\la}
 \newcommand{\psio}{\psi_1}
 \newcommand{\psii}{\psi_i}
 \newcommand{\psila}{\psi_\la}

 \newcommand{\unit}{\mathbf{1}}
 \newcommand{\tio}{\widetilde{\mathbf{1}}}
 \newcommand{\ext}[1]{\operatorname{ext}(#1)}

 \renewcommand{\t}{\bm{t}}

 \DeclareSymbolFont{largesymbol}{OMX}{yhex}{m}{n}
 \DeclareMathAccent{\widehat}{\mathord}{largesymbol}{"62}
 \newcommand{\sa}{\partial \A}
 \newcommand{\Di}{\mathbb D}
 \newcommand{\Db}{\overline{\Di}}
 \renewcommand{\H}{H^\infty(\Di)}
 
 \newcommand{\hf}{\widehat{f'\,}}
 \newcommand{\chh}{\Ch{\A}}
 \newcommand{\M}{\mathcal M}
 \newcommand{\Ao}{\mathcal{S}^\infty}
 \newcommand{\wh}[1]{\widehat{#1}}
 \renewcommand{\t}[1]{\widetilde{#1}}

 \newcommand{\BT}{B_{e}}
 \newcommand{\W}{\mathcal{W}}

 \newcommand{\Kvs}{\mathcal{K}_{e}}
 \newcommand{\KS}{\mathcal{K}_\Sigma}
 \newcommand{\KC}{\mathcal{K}_C}
 \newcommand{\Ks}{\mathcal{K}_\sigma}
 \newcommand{\DD}{\mathfrak{D}}
 \newcommand{\DS}{\DD_\Sigma}
 \newcommand{\DC}{\DD_C}
 \newcommand{\Ds}{\DD_\sigma}
 \newcommand{\Dvs}{\DD_{e}}
 
 \newcommand{\XS}{\mathcal{K}_\Sigma^\circ}

 \newcommand{\h}{\mathbf{h}}
 \newcommand{\BS}{B_\Sigma}
 \newcommand{\BTS}{\mathcal{E}_\Sigma}
 \newcommand{\B}{\mathcal{E}_e}
 \newcommand{\K}{\mathcal{K}}

\usepackage{ulem}
\usepackage{marginnote}

\begin{document}

 \def\corauthor{{Corresponding author}}


\title
[Surjective isometries on a Banach space]
{Surjective isometries on a Banach space of analytic functions with bounded derivatives}


\author[T.~Miura]{Takeshi Miura$^*$}
\address{Department of Mathematics, 
Faculty of Science, Niigata University, Niigata 950-2181  Japan}
\email{miura@math.sc.niigata-u.ac.jp}
\thanks{$^*$Corresponding author: miura@math.sc.niigata-u.ac.jp}
\thanks{The first author was supported by KAKENHI Grant Number 20K03650.}

\author[N.~Niwa]{Norio Niwa} 
\address{Research Office of Mathematics, 
School of Pharmacy, Nihon University, Chiba 274-8555  Japan} 
\email{niwa.norio@nihon-u.ac.jp}



\begin{abstract}
Let $H(\Di)$ be the linear space of all analytic functions on the open unit disc $\Di$
and $H^p(\Di)$ the Hardy space on $\Di$.
The characterization of complex linear isometries on
$\mathcal{S}^p=\set{f\in H(\Di):f'\in H^p(\Di)}$ was given
for $1 \leq p < \infty$ by Novinger and Oberlin in 1985.
Here, we characterize surjective, not necessarily linear, isometries on $\Ao$.
\end{abstract}

\maketitle


\section{Introduction and Main results}

Let $(N, \DV{\cdot}_N)$ be a real or complex normed linear space.
A mapping $T \colon N \to N$ is said to be an {\it isometry} if
$$
\V{T(f) - T(g)}_N = \V{f-g}_N
\qq (f, g \in N).
$$
Here we notice that isometries need not be linear. A linear map $S \colon N \to N$ is
an isometry if it preserves the norm in the following sense:
$\V{S(f)}_N = \V{f}_N$ for $f \in N$.
Banach \cite{ban} gave the characterization of surjective, not necessarily linear,
isometries on the Banach space
of all continuous real valued functions $C_\R(K)$, $K$ a compact metric space,
with the supremum norm. Many mathematicians have been investigating linear isometries
on various subspaces of the linear space $H(\Di)$ of all analytic functions on the open unit disc $\Di$.
Nagasawa \cite{nag} characterized surjective complex linear isometries on uniform algebras.
Let $H^p(\Di)$ be the Hardy spaces for $1\leq p\leq\infty$.
By the characterization of Nagasawa, we see that every surjective complex linear isometry
on $H^\infty(\Di)$ is a weighted composition operator.
deLeeuw, Rudin and Wermer \cite{del} obtained the form of surjective complex linear isometries
on $H^\infty(\Di)$ and $H^1(\Di)$.
Forelli \cite{for} extended the result by deLeeuw, Rudin and Wermer to the characterization of
complex linear isometries, which need not be surjective, on $H^p(\Di)$
for $1 \leq p < \infty$ with $p \neq 2$.
Results for isometries on other spaces of analytic functions were obtained
by Botelho \cite{bot}, Cima and Wogen \cite{cim}, Hornor and Jamison \cite{hor}
and Kolaski \cite{kol}.

Let $f'$ be the derivative of $f \in H(\Di)$.
Novinger and Oberlin \cite{nov} gave the characterization of complex linear isometries
on the space $\mathcal{S}^p = \set{f \in H(\Di) : f' \in H^p(\Di)}$ for
$1 \leq p < \infty$, $p \neq 2$ with respect to the following norms: $|f(0)| + \DV{f'}_{H^p}$ and
$\Vinf{f} + \DV{f'}_{H^p}$, where $\DV{\cdot}_{H^p}$ and $\DV{\cdot}_\infty$ are the usual norm
of $H^p(\Di)$ and the supremum norm on $\Di$, respectively.
In \cite{nov}, the authors assume no surjectivity of isometries on $\mathcal{S}^p$,
while they do assume that $p$ is finite.
The purpose of this paper is to characterize all surjective, not necessarily linear, isometries
on $\Ao$ with the norms $\Vinf{f} + \Vinf{f'}$, $\sup_{z\in\Di}(|f(z)|+|f'(z)|)$
and $|f(a)| + \Vinf{f'}$ for $a \in \Di$.

The following are the main results of this paper.


\begin{thm}\label{thm1}
A surjective map $T \colon \Ao \to \Ao$ is an isometry with the norms
$\VS{f} = \sup_{z \in \Di} |f(z)| + \sup_{w \in \Di} |f'(w)|$
or $\VC{f}=\sup_{z\in\Di}(|f(z)|+|f'(z)|)$
for $f \in \Ao$,
if and only if there exist constants $c, \la \in \C$ with $|c| = |\la| = 1$ such that
\begin{align*}
T(f)(z) 
&= 
T(0)(z) + cf(\la z),
&
(f \in \Ao, \ z \in \Di), & \q\mbox{or} \\
T(f)(z) 
&= 
T(0)(z) + \ov{cf(\ov{\la z})},
&
(f \in \Ao, \ z \in \Di). &
\end{align*}
\end{thm}


\begin{thm}\label{thm2}
Let $a \in \Di$.
A surjective map $T \colon \Ao \to \Ao$ is an isometry
with the norm $\Vs{f} = |f(a)| + \sup_{w \in \Di} |f'(w)|$ for $f \in \Ao$,
if and only if there exist constants $c_0, c_1, \la \in \C$ with $|c_0| = |c_1| = |\la| = 1$
and $b \in \Di$ such that,
for a function $\D \rho(z) = \la \, \fr{z-b}{1-\ov{b}z}$ defined on $\Di$,
\begin{align*}
T(f)(z) 
&= 
T(0)(z) + c_0f(a) + \int_{[a,z]} c_1 f'(\rho(\zeta))\,d\zeta,
&(f \in \Ao, \, z \in \Di), \qq\mbox{or} \\
T(f)(z) 
&= 
T(0)(z) + c_0 \ov{f(a)} + \int_{[a,z]} c_1 f'(\rho(\zeta))\,d\zeta,
&(f \in \Ao, \, z \in \Di), \qq\mbox{or}\\
T(f)(z)
&= 
T(0)(z) + c_0f(a) + \int_{[a,z]} c_1 \ov{f'(\ov{\rho(\zeta)})}\,d\zeta,
&(f \in \Ao, \, z \in \Di), \qq\mbox{or} \\
T(f)(z) 
&= 
T(0)(z) + c_0 \ov{f(a)} + \int_{[a,z]} c_1 \ov{f'(\ov{\rho(\zeta)})}\,d\zeta
&(f \in \Ao, \, z \in \Di).\qq\mbox{\phantom{or}}
\end{align*}
\end{thm}

As direct consequences of Theorems~\ref{thm1} and \ref{thm2},
we have the following results: They give the characterizations of
surjective complex linear isometries on $\Ao$ with respect to
the norms $\VS{\cdot}$ and $\Vs{\cdot}$, which had been open
since 1985 (see \cite{nov}).


\begin{cor}\label{cor1}
If $T \colon \Ao \to \Ao$ is a surjective complex linear isometry with the norms
$\VS{f}$ or $\VC{f}$,
then there exist constants $c, \la \in \C$ with $|c| = |\la| = 1$ such that
\[
T(f)(z)=cf(\la z),
\qq
(f \in \Ao, \ z \in \Di).
\]
\end{cor}


\begin{cor}\label{cor2}
Let $a \in \Di$ and $T \colon \Ao \to \Ao$ a surjective complex linear isometry
with the norm $\Vs{f}$.
There exist constants $c_0, c_1, \la \in \C$ with $|c_0| = |c_1| = |\la| = 1$
and $b \in \Di$ such that
\[
T(f)(z)=c_0f(a) + \int_{[a,z]} c_1 f'(\rho(\zeta))\,d\zeta,
\qq(f \in \Ao, \, z \in \Di),
\]
where $\D \rho(z) = \la \, \fr{z-b}{1-\ov{b}z}$ for $z\in\Di$.
\end{cor}

Theorem~\ref{thm1} states that every surjective isometry $T$ on $(\Ao,\VS{\cdot})$
or $(\Ao,\VC{\cdot})$ is a weighted composition operator, or a combination of
such operator and the complex conjugate, provided that $T(0)=0$.
The weighted composition operator is induced by constant functions of modulus $1$.
If $T$ is a surjective isometry on $(\Ao,\Vs{\cdot})$ with $T(0)=0$, then it is represented
by a combination of weighted composition operator, the complex conjugate and
the integral operator. Here, the weighted composition operator is induced by
a constant function of modulus $1$ and a M\"{o}bius transform on $\Di$.
These maps are trivial examples of surjective isometries on $\Ao$ with the norms
$\VS{\cdot}$, $\VC{\cdot}$ and $\Vs{\cdot}$.
We see that there are no other surjective isometries on the space $\Ao$
by Theorems~\ref{thm1} and \ref{thm2}.

Though our proofs of Theorems~\ref{thm1} and \ref{thm2} seem to be complicated,
the main idea of them is quite simple.
It is based on, what is so called, {\it extreme point argument}:
More explicitly, let $C(X)$ be a Banach space of all continuous complex
valued functions defined on a compact Hausdorff space $X$.
Let $S$ be an surjective, {\it complex linear} isometry from
a subspace $A$ of $C(X)$ onto itself.
Then the adjoint operator $T^*$ defined on the dual space $A^*$ of $A$
preserves extreme points of the closed unit ball $A_1^*$ of $A^*$.
It is well-known that each extreme point of $A_1^*$ is of the form
$\la\d_x$ for some $\la\in\C$ with $|\la|=1$ and $x\in\cha$,
Choquet boundary for $A$.
Here, $\d_x$ denotes the linear functional on $A$, defined by
$\d_x(f)=f(x)$ for $f\in A$.
Since $S^*(\d_x)$ is an extreme point of $A_1^*$ for $x\in\cha$,
there exist $\mu\in\C$ with $|\mu|=1$ and $y\in\cha$ such that
$S^*(\d_x)=\mu\d_y$.
We can define two mappings $\al$ and $\phi$ satisfying
$S^*(\d_x)=\al(x)\d_{\phi(x)}$,
which yields $S(f)(x)=S^*(\d_x)(f)=\al(x)\d_{\phi(x)}(f)=\al(x)f(\phi(x))$
for all $f\in A$ and $x\in\cha$.
That is, $S$ is a weighted composition operator.

To apply the above arguments to a surjective isometry $T$ on $\Ao$ in
Theorems~\ref{thm1} and \ref{thm2},
we first embed $\Ao$ into $C(K)$ for some compact Hausdorff space $K$,
which depends on the norms of $\Ao$.
In section~\ref{sect2}, we define a compact Hausdorff space $K$ and
a complex linear isometry $U\colon\Ao\to C(K)$.
Suppose that $T$ is a surjective isometry on $\Ao$.
By the Mazur-Ulam theorem, we may assume that $T$ is {\it real linear}.
Setting $A=U(\Ao)$, we define $S=U|_A\circ T\circ(U|_A)^{-1}$,
and thus $S$ is a surjective {\it real linear} isometry from $A$ onto itself.
We introduce a kind of adjoint operator, $S_*$, between $A_1^*$;
since $S$ is real linear, the adjoint operator $S^*$ is not a well defined
map on $A_1^*$.
Roughly speaking, we prove that $S_*$ preserves extreme points,
and thus it is a weighted composition operator.
In section~\ref{sect3}, we investigate the map $S=U|_A\circ T\circ(U|_A)^{-1}$.
Then we have the form of $T=(U|_A)^{-1}\circ S\circ U|_A$, which is
unsatisfactory for our purpose.
We give complete description of surjective isometry $T$ on
$\Ao$ in section~\ref{sect4}


\section{Preliminaries and an embedding of $\Ao$ into $C(K)$}
\label{sect2}

In this section, we embed the space $\Ao$ into $C(\K)$ for some
compact Hausdorff space $\K$.
Let $B$ be the isometric image of $\Ao$ in $C(\K)$.
Then we investigate extreme points of the closed unit ball of the dual space of $B$.
Such a space $\K$ depends on norm on $\Ao$, which is closely related to
Shilov boundary for a uniform algebra.
We prove, in particular, that Choquet boundary coincides with the Shilov boundary.

Let $H(\Di)$ be the algebra of all analytic functions on the open unit disc 
$\Di = \set{z \in \C : |z| < 1}$ in the complex number field $\C$. 
The constant function in $H(\Di)$, which takes the value only $1$, is denoted by $\unit$.
We denote by $\H$ the commutative Banach algebra of all bounded analytic functions
on $\Di$ with the supremum norm $\V{v}_{\Di} = \sup_{z \in \Di} |v(z)|$ for $v \in \H$.
We define a Banach space $\Ao$ by
$$
\Ao = \set{f \in H(\Di) : f' \in \H}
$$
with norms $\VS{f}$, $\VC{f}$ 
and $\Vs{f}$, given by
\[
\VS{f} = \V{f}_{\Di} + \V{f'}_{\Di},\q
\VC{f}=\sup_{z\in\Di}(|f(z)|+|f'(z)|)
\q\mbox{and}\q
\Vs{f} = |f(a)| + \V{f'}_{\Di}
\]
for $f \in \Ao$, where $a \in \Di$.
Let $A(\Db)$ be the disc algebra, that is, the commutative
Banach algebra of all analytic functions
on $\Di$ which can be extended to continuous functions on the closed unit disc $\Db$.
We see that $\Ao$ is a subset of $A(\Db)$ by \cite[Theorem~3.11]{dur}.

We need some properties of $\H$.
The maximal ideal space $\M$ of $\H$ is a compact Hausdorff space with
the relative weak*-topology of the dual space of $\H$.
The Gelfand transform $\widehat{v}$ of $v \in \H$ is a continuous function on $\M$,
defined by $\wh{v}(\eta) = \eta(v)$ for $\eta \in \M$. Then
$\V{\widehat{v}}_{\M} = \sup_{\eta \in \M} |\widehat{v}(\eta)| = \V{v}_{\Di}$
for every $v \in \H$.
Let $C(\M)$ be the commutative Banach algebra of all continuous complex valued
functions on $\M$ with the supremum norm $\V{\!\cdot\!}_\M$.
In the rest of this paper, we denote by $\A$ the uniform algebra 
$\set{\widehat{v} \in C(\M) : v \in \H}$ on $\M$, i.e., $\A$ is a uniformly closed
subalgebra of $C(\M)$
that contains the constant functions and separates the points of $\M$.
Shilov boundary $\sa$ for $\A$ is the smallest closed subset of $\M$ with
the property that $\V{\wh{v}}_{\sa} = \sup_{\eta \in \sa} |\wh{v}(\eta)| = \V{\wh{v}}_\M$
for all $\wh{v} \in \A$. 
Let $\T$ be the unit circle in $\C$. Then $\T$ is the Shilov boundary for $A(\Db)$.
Hence, $\V{\wh{f}}_\T = \sup_{z \in \T} |\wh{f}(z)| = \V{f}_{\Di}$ for all $f \in A(\Db)$.
The identity function on $\Di$ is denoted by $\id$.

We shall introduce the Choquet boundary $\Ch{\A}$ for $\A$.
Let $\A^*_1$ be the closed unit ball of the dual space of $\A$.
The Choquet boundary $\Ch{\A}$ for $\A$ is the set of all points
$\eta\in\M$ such that the point evaluation functional
$\d_\eta\colon\A\to\C$ at $\eta$ is an extreme point of $\A_1^*$.
It is well-known that
Choquet boundaries are dense in Shilov boundaries for uniform algebras.
In fact, $\Ch{\A}=\sa$; the authors of this paper believe that it is a well-known fact.
Here, we give its proof for completeness.

\begin{prop}\label{prop2.0}
The Choquet boundary for $\A$ is $\sa$.
\end{prop}

\begin{proof}
We need to show that $\sa \subset \Ch{\A}$, since $\Ch{\A}$ is dense in $\sa$.
Choose $\eta_0 \in \sa$ arbitrarily. 
Notice that the inner functions of $\H$ separate the points of $\M$
(cf. \cite[Corollary~5.2 in Chapter X]{gar}).
Let $\eta_1\in \M$ with $\eta_1\neq\eta_0$. There exists an inner function 
$f_{\eta_1}\in\H$ such that 
$\wh{f_{\eta_1}}(\eta_1)\neq\wh{f_{\eta_1}}(\eta_0)$. 
It follows from \cite[Theorem 2.2 in Chapter V]{gar} that 
$|\wh{f_{\eta_1}}|= 1$ on $\sa$. 
Then $|\wh{f_{\eta_1}}(\eta_0)|=1=\V{\wh{f_{\eta_1}}}_\M$,
since $\sa$ is a boundary for $\A$.
We set 
$\lambda_0=\wh{f_{\eta_1}}(\eta_0)\in \T$, and thus 
$\wh{f_{\eta_1}}(\eta_1)\neq\la_0$.
Define $u_{\eta_1}(z)=(\ov{\la_0}f_{\eta_1}(z)+1)/2$ for $z \in \Di$.
Then $u_{\eta_1}\in\H$ with
$\wh{u_{\eta_1}}(\eta_0)=1=\V{\wh{u_{\eta_1}}}_\M$ and
\[
|\wh{u_{\eta_1}}(\eta)|=1\quad\mbox{if and only if}\quad
\wh{u_{\eta_1}}(\eta)=1
\quad\mbox{for $\eta\in\M$}.
\]
Therefore, $\wh{u_{\eta_1}}$ is a peaking function for $\A$.
The set $P_{\wh{u_{\eta_1}}}=\set{\eta\in\M:\wh{u_{\eta_1}}(\eta)=1}$
satisfies
$\eta_0\in P_{\wh{u_{\eta_1}}}$ but $\eta_1\not\in P_{\wh{u_{\eta_1}}}$.
Hence, $\cap_{\eta_1\in\M\setminus\set{\eta_0}} P_{\wh{u_{\eta_1}}}=\set{\eta_0}$.
We observe that $\eta_0 \in \Ch{\A}$
(see, for example, \cite[Propositions~3.18 and 3.19]{hat2}), and thus
$\sa \subset\Ch{\A}$, as is claimed.
\end{proof}

The Choquet boundary $\Ch{\A}$ for $\A$ has the following property, which we shall
use later:
\begin{quote}
For each $\e>0$, $\eta_0\in\Ch{\A}$ and open neighborhood $V$ of $\eta_0$ in $\M$,
there exists $\wh{v}\in\A$ such that $\wh{v}(\eta_0)=1=\V{\wh{v}}_\M$ and
$|\wh{v}|<\e$ on $\M\setminus V$.
\end{quote}
We refer the reader to \cite{bro} for information about uniform algebras.

We may and do regard $\Di$ as a subset of $\M$. Then $\Di$ is dense in $\M$
by the Corona theorem \cite[Theorem~1.8 in Chapter V, and Chapter VIII]{gar}.
We define the fiber $\M_\xi=\set{\eta\in\M:\wh{\id}(\eta)=\xi}$ for each
$\xi\in\T$. Then we have $\M=\Di\cup\bigcup_{\xi\in\T}\M_\xi$
(see \cite[Section~1 of Chapter V]{gar}). If $f\in A(\Db)$, then $\wh{f}(\xi)=\wh{f}(\eta)$
for all $\xi\in\T$ and $\eta\in\M_\xi$ by \cite[Theorem, p.161]{hof}
(see, also \cite[Exercises and further results in Chapter V]{gar}).
Recall that the Shilov boundary $\sa$ for $\A$ is the smallest closed subset
of $\M$ so that $\V{\wh{v}}_{\sa}=\V{\wh{v}}_\M$ for all $\wh{v}\in\A$.
By the maximum modulus principle, we obtain
$\V{\wh{v}}_\M=\V{\wh{v}}_{\M\setminus\Di}$ for $v\in\H$.
Hence, $\sa\subset\M\setminus\Di$.
If, in addition, $v\in A(\Db)$, then
$\V{v}_{\Di}=\V{\wh{v}}_\M=\V{\wh{v}}_{\M\setminus\Di}=\V{\wh{v}}_\T$,
since $\wh{v}(\xi)=\wh{v}(\eta)$ for $\xi\in\T$ and $\eta\in\M_\xi$.
We refer the reader to the following books
for information about the maximal ideal space $\M$ of $\H$, \cite{gar} and \cite{hof}.

To embed $\Ao$ into $C(\Kvs)$ for some compact Hausdorff space $\Kvs$,
we introduce a compact Hausdorff space $\Dvs$, which depends on
the norms $\VS{\cdot}$, $\VC{\cdot}$ and $\Vs{\cdot}$.
Let $a \in \Di$, and we define $\DS$, $\DC$ 
and $\Ds$ by
\begin{equation}\label{D}
\DS=\T\times\sa, \q\DC=\set{(\eta,\eta):\eta\in\sa}
\qbox{and}\q
\Ds=\set{a}\times\sa.
\end{equation}
For $\Dvs\in\set{\DS,\DC,\Ds}$, we set
$$
\Kvs=\Dvs\times\T.
$$
Then $\Dvs$ and $\Kvs$ are compact Hausdorff spaces with the 
product topology. For each $f \in \Ao$, we define $\tf$ by
\begin{equation}\label{tf}
\tf(z,\eta,w) = \wh{f}(z) + \hf(\eta)w
\qq (z,\eta,w) \in \Kvs:
\end{equation}
For simplicity of notation, we shall write
$(z,\eta,w)$ instead of  $((z,\eta),w)$ for each element of $\Kvs$.
Here, we note that $f\in A(\Db)$
and $\wh{f'} \in \A$ for $f \in \Ao$.
We define the normed linear subspace $\BT$ of $C(\Kvs)$ by
\begin{equation}\label{B}
\BT = \set{\tf \in C(\Kvs) : f \in \Ao}
\end{equation}
with the supremum norm $\V{\tf}_{\Kvs} = \sup_{x \in \Kvs} |\tf(x)|$ for $\tf \in \BT$.
Note that $\V{v}_{\Di} = \V{\widehat{v}}_{\M} = \V{\wh{v}}_{\sa}$ for $v \in \H$.
Thus, we observe that
\begin{equation}\label{norm}
\V{\tf}_{\KS} = \V{\wh{f}\,}_\T + \V{\wh{f'}}_{\sa} = \VS{f}
\qbox{and}\q 
\V{\tf}_{\Ks} = |f(a)| + \V{\wh{f'}}_{\sa} = \Vs{f}
\end{equation}
for every $f \in \Ao$: In fact, since $f\in\Ao\subset A(\Db)$, we have
$\V{f}_{\Di}=\V{\wh{f}\,}_\T$ and $\V{f'}_{\Di}=\V{\wh{f'}}_{\sa}$.
Hence, $\VS{f}=\V{f}_{\Di}+\V{f'}_{\Di}=\V{\wh{f}\,}_{\T}+\V{\wh{f'}}_{\sa}$.
Let $(\xi_0,\eta_0)\in\T\times\sa$ be such that
$|\wh{f}(\xi_0)|=\V{\wh{f}\,}_\T$ and $|\wh{f'}(\eta_0)|=\V{\wh{f'}}_{\sa}$.
Such a pair of points do exist, since $\wh{f}$ and $\wh{f'}$ are continuous
on the compact sets $\T$ and $\sa$, respectively.
Choose $w_0\in\T$ so that
$|\wh{f}(\xi_0)+\wh{f'}(\eta_0)w_0|=|\wh{f}(\xi_0)|+|\wh{f'}(\eta_0)|$.
We obtain
\[
\VS{f}=
\V{\wh{f}\,}_\T+\V{\wh{f'}}_{\sa}
=|\wh{f}(\xi_0)+\wh{f'}(\eta_0)w_0)|
\leq
\V{\tf}_{\KS}
\leq
\V{\wh{f}\,}_\T+\V{\wh{f'}}_{\sa}
=\VS{f},
\]
which prove the first equalities in \eqref{norm}.
We can prove the second equalities in \eqref{norm} by a similar argument
to the above, and we omit it.

We shall prove that $\V{\tf}_{\KC}= \VC{f}$ for every $f\in\Ao$.


\begin{prop}\label{prop2.1}
For each $f \in \Ao$,
\begin{equation}\label{prop2.1.1}
\V{\tf}_{\KC} = \sup_{(\eta,\eta,w)\in\KC} |\wh{f}(\eta) + \hf(\eta)w|
= \VC{f}.
\end{equation}
\end{prop}

\begin{proof}
Let $f\in\Ao$ and $z\in\Di$.
There exists $w_z\in\T$ such that
$|\wh{f}(z)|+|\wh{f'}(z)|=|\wh{f}(z)+\wh{f'}(z)w_z)|$.
Since $f,f'w_z\in\H$, we have
$\wh{f+f'w_z}=\wh{f}+\wh{f'}w_z$ and
\[
|\wh{f}(z)|+|\wh{f'}(z)|=|\wh{f}(z)+\wh{f'}(z)w_z|
\leq\sup_{\eta\in\sa}|\wh{f}(\eta)+\wh{f'}(\eta)w_z|
\leq\V{\tf}_{\KC}.
\]
Since $z\in\Di$ is arbitrarily chosen, we obtain
$\VC{f}\leq\V{\tf}_{\KC}$.

By the continuity of $\tf$, there exists $\eta_0\in\sa$ and $w_0\in\T$
such that $|\wh{f}(\eta_0)+\wh{f'}(\eta_0)w_0|=\V{\tf}_{\KC}$.
Since $\Di$ is dense in $\M$ by the Corona theorem, for each $\e>0$
there exists $z_0\in\Di$ such that
$|\tf(\eta_0,\eta_0,w_0)-\tf(z_0,z_0,w_0)|<\e$.
Then we obtain
\[
\V{\tf}_{\KC}=|\wh{f}(\eta_0)+\wh{f'}(\eta_0)w_0|
=|\tf(\eta_0,\eta_0,w_0)|
\leq|\tf(z_0,z_0,w_0)|+\e=|\wh{f}(z_0)+\wh{f'}(z_0)w_0|+\e,
\]
and therefore, $\V{\tf}_{\KC}\leq\VC{f}+\e$.
Since $\e>0$ is arbitrarily chosen, we get $\V{\tf}_{\KC}\leq\VC{f}$,
and consequently, $\V{\tf}_{\KC}=\VC{f}$.
\end{proof}

Let $\Kvs\in\set{\KS,\KC,\Ks}$.
For each $x \in \Kvs$, we define the point 
evaluation functional $\d_x \colon \BT \to \C$
at $x$ by
$$
\d_x(\tf) = \tf(x)
\qq (\tf \in \BT).
$$
Define the complex linear operator $I_0\colon\H\to\H$ by
\begin{equation}\label{I}
I_0(v)(z) = \int_{[0,z]} v(\zeta)\,d\zeta
\qq (v\in\H, \ z \in \Di),
\end{equation}
where $[0,z]$ is the straight line interval from $0$ to $z$.

We shall show that $\BT$ separates the points of $\Kvs$
in the following two propositions, in which the operator
$I_0$ plays an important role.


\begin{prop}\label{prop2.2}
Let $\Kvs\in\set{\KS,\Ks}$.
For each pair of distinct points $x_1, x_2 \in \Kvs$
there exists $\tf \in \BT$ such that $\tf(x_1) \neq \tf(x_2)$.
\end{prop}

\begin{proof}
Let $x_j = (z_j, \eta_j, w_j) \in \Kvs$ for $j=1,2$ with $x_1 \neq x_2$. 
Suppose first that $z_1 \neq z_2$. We  set $\xi_j = \wh{\id}(\eta_j)$
for $j=1,2$.
Then $\xi_j\in\T$, since $\eta_j\in\sa$.
We have three possible cases to consider.

{\bf Case 1.}
$\xi_1, \xi_2 \not\in \set{z_1,z_2}$.
Let $f(z) = (z-z_2)(z-\xi_1)^2(z-\xi_2)^2 \in \Ao$. Then we obtain
$\wh{f}(z_1) \neq 0 = \wh{f}(z_2)$ and $\wh{f'}(\xi_1) = 0 = \wh{f'}(\xi_2)$.
Equality \eqref{tf} shows that $\tf(x_1) = \wh{f}(z_1) \neq 0 = \tf(x_2)$.

{\bf Case 2.}
$\xi_j \in \set{z_1,z_2}$ and $\xi_k \not\in \set{z_1,z_2}$
for $j, k \in \set{1,2}$ with $j \neq k$.
Without loss of generality, we may and do assume that
$\xi_1 \in \set{z_1,z_2}$ and $\xi_2 \not\in \set{z_1,z_2}$. 
For functions $g_1(z) = (2z-3z_1+z_2)(z-z_2)^2 \in \Ao$
and $g_2(z) = (2z-3z_1+\xi_2)(z-\xi_2)^2 \in \Ao$, we see that
\[
\wh{g_1}(z_1) \neq 0 = \wh{g_1}(z_2) = \wh{g_1'}(z_1) = \wh{g_1'}(z_2)
\q\mbox{and}\q
\wh{g_2}(z_1) \neq 0 = \wh{g_2}(\xi_2) = \wh{g_2'}(z_1) = \wh{g_2'}(\xi_2).
\]
We set $g = g_1g_2 \in \Ao$, and then $\wh{g}(z_1) \neq 0 = \wh{g}(z_2)$.
The multiplication law, $g' = {g_1}' g_2 + g_1 {g_2}'$, shows that
$\wh{g'}(z_1) = 0 = \wh{g'}(z_2) = \wh{g'}(\xi_2)$.
Because $\xi_1 \in \set{z_1,z_2}$, we get $\wh{g'}(\xi_1)= 0$.
Since $g'\in A(\Db)$, we have $\wh{g'}(\xi)=\wh{g'}(\eta)$
for all $\xi\in\T$ and $\eta\in\M_\xi$.
Therefore, $\wh{g'}(\eta_j) = \wh{g'}(\xi_j) = 0$ for $j=1,2$.
It follows from \eqref{tf} that $\widetilde{g}(x_1) = \wh{g}(z_1) \neq 0 = \widetilde{g}(x_2)$.

{\bf Case 3.}
$\xi_1, \xi_2 \in \set{z_1,z_2}$.
Let $h(z) = (2z-3z_1+z_2)(z-z_2)^2 \in \Ao$, and then we see that
$\wh{h}(z_1) \neq 0 = \wh{h}(z_2)$ and $\wh{{h}'}(z_1) = 0 = \wh{{h}'}(z_2)$.
By the choice of $\xi_j$, we have $\wh{{h}'}(\xi_j) = 0$ for $j=1,2$.
Equality \eqref{tf} shows that $\t{h}(x_1) = \wh{h}(z_1) \neq 0 = \t{h}(x_2)$.

Now we assume that $z_1 = z_2$ and $w_1 \neq w_2$.
We have that
$\t{\id}(x_1) = z_1 + w_1 \neq z_2 + w_2 = \t{\id}(x_2)$ by \eqref{tf}.

If $z_1 = z_2$ and $w_1 = w_2$, then $\eta_1 \neq \eta_2$.
Since the image $\A$ of the Gelfand transform of $\H$ separates the points of $\M$,
there exists $v \in \H$ such that $\wh{v}(\eta_1) \neq \wh{v}(\eta_2)$.
Because $z_1 = z_2$ and $w_1 = w_2$, we see that $\t{I_0(v)}(x_1) \neq \t{I_0(v)}(x_2)$
by \eqref{tf} and \eqref{I}.

Consequently, $\tf(x_1) \neq \tf(x_2)$ for some $\tf \in \BT$, as is claimed.
\end{proof}


\begin{prop}\label{prop2.3}
For each pair of distinct points $x_1, x_2\in\KC$
there exists $\tf \in B_C$ such that $\tf(x_1) \neq \tf(x_2)$.
\end{prop}

\begin{proof}
Let $x_j = (\eta_j,\eta_j,w_j)\in\KC$ for $j=1,2$ with $x_1 \neq x_2$. 
Suppose that $\eta_1 \neq \eta_2$ and $\wh{\id}(\eta_1) = \wh{\id}(\eta_2)$.
Since $\A$ separates the points of $\M$, we can choose $v_0 \in \H$ so that
$\wh{v_0}(\eta_1) = 1$ and $\wh{v_0}(\eta_2) = 0$.
We set $f_0 = I_0(v_0)$, and then $f_0 \in A(\Db)$.
Since $\wh{\id}(\eta_1) = \wh{\id}(\eta_2)$ and $f_0 \in A(\Db)$,
we have $\wh{f_0}(\eta_1) = \wh{f_0}(\eta_2)$ with
$\wh{{f_0}'}(\eta_1) = 1$ and $\wh{{f_0}'}(\eta_2) = 0$.
Thus, $\t{f_0}(x_1) = \wh{f_0}(\eta_1) + w_1 \neq \wh{f_0}(\eta_2) = \t{f_0}(x_2)$
by \eqref{tf}.
We now assume that $\eta_1 \neq \eta_2$ and $\wh{\id}(\eta_1) \neq \wh{\id}(\eta_2)$.
We set $\xi_j = \wh{\id}(\eta_j)$ for $j=1,2$, and
$f_1(z) = (2z-3\xi_1+\xi_2)(z-\xi_2)^2 \in A(\Db)$.
Then we get $\wh{f_1}(\xi_1) \neq 0 = \wh{f_1}(\xi_2)$ and
$\wh{f_1'}(\xi_1) = \wh{f_1'}(\xi_2) = 0$. Therefore,
$\t{f_1}(x_1) \neq 0 = \t{f_1}(x_2)$ by \eqref{tf}.

Suppose now that $\eta_1 = \eta_2$, and then $w_1 \neq w_2$ by the choice of $x_j$'s.
If we take $f_2 = \id \in \Ao$, then we see that $\t{f_2}(x_1) \neq \t{f_2}(x_2)$.
\end{proof}

We shall investigate Choquet boundary for $\BT$, which is deeply connected
with the set of all extreme points of the closed unit ball of the dual space
of $\BT$.
To this end, we need some topological properties of the maximal ideal space
$\M$ of $\H$.
Essential ideas of proof of Lemmas \ref{lem2.3} and \ref{lem2.4} below are due to
Professor Osamu Hatori. The authors are grateful for his suggestions and comments
on them.


\begin{lem}\label{lem2.3}
Let $W_0 = \set{z \in \Di : |z| < 1-\e_0}$ and $W_1 = \set{z \in \Di : |z|>1-\e_0}$
for each $\e_0$ with $0 < \e_0 < 1$.
Then $\M \setminus W_0$ is the weak*-closure of $W_1$ in $\M$.
\end{lem}

\begin{proof}
Let $\cl{W_1}$ be the weak*-closure of $W_1$ in $\M$.
Since $\M \setminus W_0$ is a closed subset of $\M$ containing $W_1$,
we see that $\cl{W_1} \subset \M \setminus W_0$.

Now we shall prove that $\M \setminus \Di \subset \cl{W_1}$.
Let $\eta_0 \in \M \setminus \Di$.
By the Corona theorem,
the open unit disc $\Di$ is dense in $\M$, 
and then there exists a net $\set{z_\nu}$ in $\Di$ such that
$\set{z_\nu}$ converges to $\eta_0$ with respect to the weak*-topology.
We assert $\eta_0 \in \cl{W_1}$; in fact, if $\eta_0$ were not in $\cl{W_1}$,
then we would have that $\eta_0$ is in the open set $\M \setminus \cl{W_1}$.
Since $\set{z_\nu}$ converges to $\eta_0$, we may and do assume that
$z_\nu \in \M \setminus \cl{W_1}$ for all $\nu$.
Then we get $z_\nu \in \Di \setminus W_1$ for all $\nu$.
Let $\M_\xi$ be the fiber over $\xi \in \T$.
Because $\eta_0 \in \M \setminus \Di = \cup_{\xi \in \T} \M_\xi$,
there exists a unique $\xi_0 \in \T$ such that $\eta_0 \in \M_{\xi_0}$.
As $z_\nu\in\Di\setminus W_1$, we obtain $|z_\nu| \leq 1-\e_0$,
and then $|z_\nu - \xi_0| \geq \e_0$ for all $\nu$.
Since $\set{z_\nu}$ converges to $\eta_0$ with weak*-topology,
$z_\nu = \wh{\id}(z_\nu)$ tends to $\wh{\id}(\eta_0) = \xi_0$, which contradicts
$|z_\nu - \xi_0| \geq \e_0$ for all $\nu$.
This implies $\eta_0 \in \cl{W_1}$, and
consequently $\M \setminus \Di \subset \cl{W_1}$.

Notice that $\set{z \in \Di : |z|\geq 1-\e_0} \subset \cl{W_1}$.
As $\M \setminus \Di \subset \cl{W_1}$, we get
$\M \setminus W_0 = \set{z \in \Di : |z|\geq 1-\e_0} \cup (\M \setminus \Di)
\subset \cl{W_1}$.
Thus $\cl{W_1} = \M \setminus W_0$ as is claimed.
\end{proof}


\begin{lem}\label{lem2.4}
Let $W_1=\set{z\in\Di:|z|>1-\e_0}$ for each $\e_0$ with $0<\e_0<1$.
Then $\sa$ is contained in the interior of $\cl{W_1}$.
\end{lem}

\begin{proof}
Let $W_0 = \set{z \in \Di : |z| < 1-\e_0}$.
Then $\cl{W_1} = \M \setminus W_0$ by Lemma~\ref{lem2.3}. Because
$\sa \subset \M \setminus \Di$,
we deduce $\sa\subset\cl{W_1}$. Let $\eta_0\in\sa$.
We shall prove that $\eta_0$ is an interior point of $\cl{W_1}$.
If any open neighborhood $O$ of $\eta_0$ were not contained in $\cl{W_1}$,
there would exist a net $\set{\eta_\nu}$ in $\M \setminus \cl{W_1}$ such that
$\set{\eta_\nu}$ converges to $\eta_0$ with respect to the weak*-topology.
Since $\cl{W_1} = \M \setminus W_0$,
we get $\eta_\nu \in W_0$, and hence $|\eta_\nu| < 1-\e_0$ for all $\nu$.
By the choice of $\set{\eta_\nu}$ and $\eta_0$, we obtain
$\eta_\nu = \wh{\id}(\eta_\nu) \to \wh{\id}(\eta_0)$.
Here, $|\wh{\id}(\eta_0)| = 1$, since $\eta_0 \in \sa \subset \M \setminus \Di$.
Therefore, $|\eta_\nu|\to 1$, which contradicts $|\eta_\nu| < 1 - \e_0$ for all $\nu$.
Thus, $\eta_0$ is an interior point of $\cl{W_1}$.
\end{proof}

In the next lemma, we shall prove the existence of ``peaking functions"
in $\Ao$, which will be significant to determine the set of all extreme points
the closed unit ball of the dual space of $\BT$.
We need topological properties of $\M$ as in Lemmas~\ref{lem2.3} and
\ref{lem2.4} to prove the next lemma.


\begin{lem}\label{lem2.5}
Let $\e > 0$, $\eta_0 \in \sa$, and let $V$ be an open neighborhood of $\eta_0$ in $\M$.
There exists $f \in \Ao$ such that
$$
\V{f}_\Di \leq \e, \q
\wh{f'}(\eta_0) = 1 = \V{\wh{f'}}_\M
\qbox{and}\q
|\wh{f'}| < \e
\qbox{on}\q
\M \setminus V.
$$
\end{lem}

\begin{proof}
Let $\e > 0$, $\eta_0 \in \sa$ and $V$ an open neighborhood of $\eta_0$ in $\M$.
Let $2\e_0 = \min\set{\e, 1}$, and we set $W_0 = \set{z \in \Di : |z| < 1-\e_0}$ and 
$W_1 = \set{z \in \Di :|z|>1-\e_0}$.
By Lemma~\ref{lem2.4}, $\eta_0$ is an interior point of $\cl{W_1}$.
Thus, there exists an open neighborhood $O_1$ of $\eta_0$ in $\M$
such that $O_1 \subset \cl{W_1}$.
We set $V_1 = V \cap O_1$, and then $V_1$ is an open neighborhood of $\eta_0$
with $V_1\subset\cl{W_1}$.
Since $\cl{W_1}=\M\setminus W_0$ by Lemma~\ref{lem2.3}, we see that
$V_1\cap\Di\subset\set{z\in\Di:|z|\geq 1-\e_0}$.
Notice that
$\eta_0 \in \sa = \chh$, the Choquet boundary for $\A$,
by Proposition~\ref{prop2.0}.
There exists $v \in \H$ such that $\wh{v}(\eta_0) = 1 = \V{\wh{v}}_\M$
and $|\wh{v}| < \e_0$ on $\M \setminus V_1$.
Define $f \in \Ao$ by $f = I_0(v)$, and then $f' = v$ on $\Di$
and $\wh{f'}(\eta_0) = 1 = \V{\wh{f'}}_\M$.
Since $V_1 \subset V$, we obtain $|\wh{f'}| < \e$ on $\M \setminus V$. 

We show that $\V{f}_\Di \leq \e$.
Let $z_0$ be an arbitrary point of $\Di$.
If $z_0 \in W_0 = \set{z \in \Di : |z| < 1-\e_0}$, then the line interval
$[0,z_0]$ from $0$ to $z_0$ lies in $W_0$.
Since $V_1 \cap \Di \subset \set{z \in \Di : |z|\geq 1-\e_0}$, 
it follows that $|v| < \e_0$ on $[0,z_0]$. Thus,
$$
|f(z_0)| = |I_0(v)(z_0)| = \left| \int_{[0,z_0]} v(\zeta)\,d\zeta \right|
\leq |z_0| \sup_{\zeta \in [0,z_0]} |v(\zeta)| < \e_0.
$$
Hence, $|f(z_0)| < \e_0$ for each $z_0 \in W_0$.

If $z_0 \in \Di \setminus W_0$, then we set $z_r = (1-r\e_0) z_0/|z_0|$
for each $r$ with $1 < r < 2$.
Note that $\e_0 < r\e_0 < 2\e_0 \leq 1$. 
Because $|z_r| = 1-r\e_0 < 1-\e_0$, we obtain $z_r \in W_0$.
Then $|f(z_r)| \leq \e_0$ by the fact proved in the last paragraph.
By the choice of $z_r$ and $z_0$, we get $|z_r - z_0| < r\e_0$.
It follows that
\begin{align*}
|f(z_0)|
&=
|I_0(v)(z_0)| \leq \left| \int_{[0,z_r]} v(\zeta)\,d\zeta \right|
+ \left| \int_{[z_r, z_0]} v(\zeta)\,d\zeta \right| \\
&\leq
|f(z_r)| + |z_r - z_0| \sup_{\zeta \in [z_r,z_0]} |v(\zeta)|
\leq \e_0 + r\e_0 = (1+r)\e_0.
\end{align*}
Letting $r\searrow1$, we obtain $|f(z_0)| \leq 2\e_0 \leq \e$
for $z_0 \in \Di \setminus W_0$.

We thus conclude that $|f| \leq \e$ on $\Di$,
and consequently $\V{f}_\Di \leq \e$.
\end{proof}

Recall that $\BT$ is a normed linear subspace of $C(\Kvs)$, defined by \eqref{B}.
Let $\Lambda\colon\BT\to\C$ be a bounded linear functional with the
operator norm $\V{\Lambda}$.
By the Hahn-Banach theorem, $\Lambda$ is extended to a bounded
linear functional $\Lambda_1$ on $C(\Kvs)$ with $\V{\Lambda}=\V{\Lambda_1}$.
Thus, the Riesz representation theorem ensures that there exists a regular
Borel measure $\mu$ on $\Kvs$ such that $\D\Lambda(\tf)=\int_{\Kvs}\tf\,d\mu$
for every $\tf\in\BT$ and $\V{\Lambda}=\V{\mu}$, where $\V{\mu}$ denotes
the total variation of the measure $\mu$.

We investigate representing measures for the point evaluation functional $\d_x$
with $x\in\Kvs$ to determine the set of extreme points of the closed unit ball
of $\BT$.
We first show that any representing measure $\mu$ for $\d_x$ with
$x\in\KC$ or $x\in\Ks$ is a Dirac measure.
To this end, we shall prove that any such measure $\mu$ is concentrated
on some sets.

Here, we recall that $\DS=\T\times\sa$, $\DC=\set{(\eta,\eta):\eta\in\sa}$ and
$\Ds=\set{a}\times\sa$, $a\in\Di$ by \eqref{D}.
We denote by $\pi$ the natural projection from $\Dvs$ onto $\sa$, that is,
$\pi(z,\eta)=\eta$ for $(z,\eta)\in\Dvs$.


\begin{lem}\label{lem2.6}
Let $\Kvs\in\set{\KS,\KC,\Ks}$, $x_0=(z_0,\eta_0,w_0)\in\Kvs$
and let $V$ be an open neighborhood of $\eta_0$ in $\M$.
Then $\mu(\pi^{-1}(V) \times \T) = 1$
for each representing measure $\mu$ for $\d_{x_0}$.
\end{lem}

\begin{proof}
Let $x_0 = (z_0, \eta_0, w_0)$ be an arbitrary point of $\Kvs$.
Let $\mu$ be a representing measure for $\d_{x_0}$.
Since $\d_{x_0}(\t{\unit}) = 1 = \V{\d_{x_0}}$, we see that $\mu$ is a probability measure
(see, for example, \cite[p. 81]{bro}).

Let $\e > 0$ be an arbitrary positive real number and $V$ an open neighborhood 
of $\eta_0 \in \sa$ in $\M$.
We set $V^c = \sa \setminus V$,
$$
P_{V} = \pi^{-1}(V)\times \T
\q\mbox{and}\q
P_{V^c} = \pi^{-1}(V^c)\times \T.
$$
Then we obtain
\begin{equation}\label{lem2.3.1}
(z_0, \eta_0, w_0) \in P_V, \q
P_V \cup P_{V^c} = \Kvs
\q\mbox{and}\q
P_V \cap P_{V^c} = \emptyset.
\end{equation}
By Lemma~\ref{lem2.5}, there exists $f_0 \in \Ao$ such that
$$
\V{f_0}_\Di \leq \e, \q \wh{f_0'}(\eta_0) = 1 = \V{\wh{f_0'}}_\M \q\mbox{and}\q
|\wh{f_0'}| < \e \qbox{on}\q V^c.
$$
Recall that $\t{f_0}(z, \eta, w) = \wh{f_0}(z) + \wh{f_0'}(\eta)w$ by \eqref{tf}.
By the choice of $f_0$, we have
$$
|\t{f_0}| \leq \e + 1 \q\mbox{on} \ P_V \qq\mbox{and}\qq
|\t{f_0}| < 2\e \q\mbox{on} \ P_{V^c}.
$$
Since $\mu$ is a representing measure for $\d_{x_0}$, we get
$$
\int_{\Kvs} \t{f_0}\,d\mu
= \d_{x_0}(\t{f_0}) = \wh{f_0}(z_0) + \wh{f_0'}(\eta_0)w_0
= \wh{f_0}(z_0) + w_0. 
$$
We derive from the above equalities with \eqref{lem2.3.1} that
\[
1-\e 
\leq|\wh{f_0}(z_0)+w_0| = 
\left| \int_{\Kvs} \t{f_0}\,d\mu \right| 
\leq \left| \int_{P_V} \t{f_0}\,d\mu \right| + \left| \int_{P_{V^c}} \t{f_0}\,d\mu \right|
<(\e + 1)\mu(P_V) + 2\e\mu(P_{V^c}).
\]
Since $\e>0$ is arbitrary, we get $1 \leq \mu(P_V) \leq \mu(\Kvs) = 1$. 
Hence $1 = \mu(P_V) = \mu(\pi^{-1}(V)\times\T)$.
\end{proof}


\begin{lem}\label{lem2.7}
Let $\Kvs\in\set{\KS,\KC,\Ks}$ and $x_0=(z_0,\eta_0,w_0)\in\Kvs$.
Then $\mu(\pi^{-1}(\set{\eta_0}) \times \T) = 1$
for each representing measure $\mu$ for $\d_{x_0}$.
\end{lem}

\begin{proof}
Let $P$ be an open set in $\Kvs=\Dvs\times\T$
with $\pi^{-1}(\set{\eta_0})\times\T\subset P$.
For each $(z,\eta_0,w)\in\pi^{-1}(\set{\eta_0})\times\T$, there exist open sets
$Z, V, W$ such that
$(z,\eta_0,w)\in Z\times V\times W\subset P$ by the definition
of the product topology. Note that $\pi^{-1}(\set{\eta_0})\times\T$ is a
compact set, and then there exist finitely many open sets
$Z_k,V_k,W_k$ such that $\pi^{-1}(\set{\eta_0})\times\T\subset\cup_{k=1}^n
(Z_k\times V_k\times W_k)$ for some $n\in\N$. We now set $V_0=\cap_{k=1}^n V_k$.
Then $V_0$ is an open neighborhood of $\eta_0$. We show that
$\pi^{-1}(V_0)\times\T\subset P$. In fact, if $(z_1,\eta_1,w_1)\in\pi^{-1}(V_0)\times\T$,
then $(z_1,\eta_0,w_1)\in\pi^{-1}(\set{\eta_0})\times\T$.
By the choice of $Z_k,V_k,W_k$, we obtain
$(z_1,\eta_0,w_1)\in Z_m\times V_m\times W_m$
for some $m$ with $1\leq m\leq n$.
We assure from $\eta_1\in V_0\subset V_m$ that
$(z_1,\eta_1,w_1)\subset Z_m\times V_m\times W_m\subset P$,
which implies $\pi^{-1}(V_0)\times\T\subset P$, as is claimed.
We thus have that $1=\mu(\pi^{-1}(V_0)\times\T)\leq\mu(P)$
by Lemma~\ref{lem2.6}.
Since $\mu$ is a probability measure, we get $\mu(P)=1$.
Since $P$ is an arbitrary open set
with $\pi^{-1}(\set{\eta_0})\times\T\subset P$,
the regular measure $\mu$ satisfies $\mu(\pi^{-1}(\set{\eta_0})\times\T)=1$.
\end{proof}

\begin{rem}\label{rem2.1}
If $\Dvs=\DC=\set{(\eta,\eta):\eta\in\sa}$,
then $\pi^{-1}(\set{\eta_0})=\set{(\eta_0,\eta_0)}$.
If $\Dvs=\Ds=\set{a}\times\sa$, then
$\pi^{-1}(\set{\eta_0})=\set{(a,\eta_0)}$.
Lemma~\ref{lem2.7} states that if $x_0=(z_0,\eta_0,w_0)\in\Kvs$
and if $\Kvs\in\set{\KC,\Ks}$,
then $\mu(\set{z_0}\times\set{\eta_0}\times\T)=1$ for every
representing measure $\mu$ for $\d_{x_0}$.
The next lemma says that a similar result is true even if $\Kvs=\KS$.
\end{rem}


\begin{lem}\label{lem2.8}
If $x_0 = (z_0, \eta_0, w_0) \in\KS=\T \times \sa \times \T$ with
$\eta_0 \not\in\M_{z_0}$
then $\mu(\set{z_0} \times \set{\eta_0} \times \T) = 1$
for each representing measure $\mu$ for $\d_{x_0}$.
\end{lem}

\begin{proof}
Let $\e > 0$ and $Z$ an open neighborhood of $z_0$ in $\T$.
We set $Z^c = \T \setminus Z$, $Q_Z = Z \times \set{\eta_0} \times \T$ and
$Q_{Z^c} = Z^c \times \set{\eta_0} \times \T$. Then we see that
$$
Q_Z \cup Q_{Z^c} = \T \times \set{\eta_0} \times \T
\q\mbox{and}\q
Q_Z \cap Q_{Z^c} = \emptyset.
$$
Notice that $\pi^{-1}(\set{\eta_0})=\T\times\set{\eta_0}$,
and hence $\mu(Q_Z \cup Q_{Z^c}) = 1$ by Lemma~\ref{lem2.7}.
For each $m \in \N$, we define a function $g_m \in \Ao$ by
$g_m(z) = \{ (\ov{z_0}\,z + 1)/2 \}^m$ for $z \in \Di$.
We observe that
$$
\wh{g_m}(z_0) = 1 = \V{\wh{g_m}}_{\T}
\qbox{and}\q
|\wh{g_m}(z)| < 1 \qbox{for $z \in \T \setminus \set{z_0}$}.
$$
Set $\xi_0 = \wh{\id}(\eta_0) \in \T$.
By the assumption, $\eta_0\not\in\M_{z_0}$.
Thus $\xi_0 \neq z_0$,
and we get $|(\ov{z_0}\,\xi_0+1)/2| < 1$. Therefore,
$m\{(\ov{z_0}\,\xi_0+1)/2\}^{m-1}$ converges to $0$ as $m \to \infty$.
In addition, since $|\wh{g_m}(z)| < 1$ for all $z \in \T \setminus \set{z_0}$,
we can find $m_0 \in \N$ such that
$|\wh{g_{m_0}}| < \e$ on $Z^c$ and
$$
|\wh{{g_{m_0}}'}(\eta_0)| = 
\left| \fr{m_0 \ov{z_0}}{2} \left( \fr{\ov{z_0}\,\xi_0 + 1}{2} \right)^{m_0-1} \right|
< \e.
$$
We set $f_1 = g_{m_0} \in \Ao$, and hence
$$
\wh{f_1}(z_0) = 1 = \V{\wh{f_1}}_{\T}, \q
|\wh{f_1}| < \e \q\mbox{on}\ Z^c
\q\mbox{and}\q
|\wh{f_1'}(\eta_0)| < \e.
$$
Equality \eqref{tf} with the above shows that
$|\d_{x_0}(\t{f_1})| = |\wh{f_1}(z_0) + \wh{f_1'}(\eta_0)w_0| \geq 1 - \e$.
We deduce from $\mu(Q_Z \cup Q_{Z^c}) = 1$ that
\begin{align*}
1-\e
&\leq
|\d_{x_0}(\t{f_1})| = \left| \int_{Q_Z \cup Q_{Z^c}} \t{f_1}\,d\mu \right|\\
&\leq
\left| \int_{Q_Z} \left\{ \wh{f_1}(z) + \wh{f_1'}(\eta_0)w \right\} \,d\mu \right|
+ \left| \int_{Q_{Z^c}} \left\{ \wh{f_1}(z) + \wh{f_1'}(\eta_0)w \right\} \,d\mu \right|\\
&\leq
(1+\e)\mu(Q_Z) + 2\e\mu(Q_{Z^c}).
\end{align*}
Because $\e>0$ is arbitrarily chosen, we get $1 \leq \mu(Q_Z)$.
Since $\mu$ is a probability measure,
$1 = \mu(Q_Z) = \mu(Z \times \set{\eta_0} \times \T)$ for all
open neighborhood $Z$ of $z_0$.
By a similar argument to proof of Lemma~\ref{lem2.7}, we observe that
$\mu(\set{z_0} \times \set{\eta_0} \times \T) = 1$.
\end{proof}

Now we are ready to characterize representing measure
for $x_0\in\Kvs$.
Applying Lemmas~\ref{lem2.7} and \ref{lem2.8}, we shall prove that
the representing measure for $\d_{x_0}$ is concentrated on
the point $x_0\in\Kvs$.


\begin{lem}\label{lem2.9}
Let $x_0 = (z_0, \eta_0, w_0) \in \Kvs$
with $\Kvs\in\set{\KS,\KC,\Ks}.$
 Assume further that
$\eta_0 \not\in \M_{z_0}$ if $x_0\in\KS$.
Then the Dirac measure concentrated at $x_0$ is the unique
representing measure for $\d_{x_0}$.
\end{lem}

\begin{proof}
Let $x_0 = (z_0, \eta_0, w_0) \in \Kvs$. We assume that
$\eta_0 \not\in \M_{z_0}$ if $x_0\in\KS$.
Let $\mu$ be a representing measure for $\d_{x_0}$.
We shall prove that $\mu(\set{x_0}) = 1$.
Set $R = \set{z_0} \times \set{\eta_0} \times \T$,
and thus $\mu(R)=1$ by Lemma~\ref{lem2.7} (see, also Remark~\ref{rem2.1})
and Lemma~\ref{lem2.8}.

Let $f_0 = \id-\wh{\id}(z_0)\unit \in \Ao$.
Then $\wh{f_0}(z_0) = 0$ and $\wh{f_0'}(\eta_0) = 1$.
By \eqref{tf},
$\d_{x_0}(\t{f_0}) = \wh{f_0}(z_0) + \wh{f_0'}(\eta_0)w_0 = w_0$.
Since $\mu$ is a representing measure for $\d_{x_0}$, 
the fact $\mu(R) = 1$ shows that
\begin{align*}
w_0
&= 
\int_R \t{f_0}\,d\mu = \int_R \left\{ \wh{f_0}(z_0) + \wh{f_0'}(\eta_0)w \right\} d\mu 
= \int_R w\,d\mu.
\end{align*}
Hence, $\D \int_R (w-w_0)\,d\mu =0$. Multiplying
this equality by $-\ov{w_0}$, we have
$\D \int_R (1-\ov{w_0}w)\,d\mu = 0$.
Since $\mu$ is a positive measure,
$$
\int_R (1-\Re (\ov{w_0}w))\,d\mu = \Re \int_R (1-\ov{w_0}w)\,d\mu = 0,
$$
where $\Re\zeta$ is the real part of a complex number $\zeta$.
Here, we note that $1-\Re(\ov{w_0}w)> 0$ on
$R_0=\set{z_0}\times\set{\eta_0}\times(\T\setminus\set{w_0})\subset R$.
We thus obtain 
$\mu(R_0)=0$, and hence
$\mu(\set{z_0} \times \set{\eta_0} \times \set{w_0})=\mu(R)=1$.
This implies that $\mu(\set{x_0}) = 1$.
We conclude that $\mu$ is a Dirac measure concentrated at $x_0 = (z_0, \eta_0, w_0)$,
as is claimed. 
\end{proof}

Let $\chb$ be the Choquet boundary for $\BT$, that is, the set of all
$x \in \Kvs$ such that
$\d_x$ is an extreme point of $(\BT)_1^*$, the closed unit ball of the dual space
$(\BT)^*$ of $(\BT, \V{\cdot}_{\Kvs})$. 
We denote by $\ext{(\BT)_1^*}$ the set of all extreme points of $(\BT)_1^*$.
By the Arens-Kelley theorem, we see that
$\ext{(\BT)_1^*} = \set{\la \d_x : \la \in \T, \ x \in \chb}$
(cf. \cite[Corollary~2.3.6 and Theorem~2.3.8]{fle1}).

The next lemma gives the characterization of $\chb$.
In other words, we can determine $\ext{(\BT)_1^*}$, which is essential
to our arguments below.


\begin{lem}\label{lem2.10}
Let $x_0 = (z_0, \eta_0, w_0) \in \Kvs$
with $\Kvs\in\set{\KS,\KC,\Ks}$. Assume that
$\eta_0 \not\in \M_{z_0}$ if $x_0\in\KS$.
Then $x_0\in\chb$.
\end{lem}

\begin{proof}
Let $x_0 = (z_0, \eta_0, w_0) \in\Kvs$, and $\eta_0 \not\in \M_{z_0}$
if $x_0\in\KS$.
We prove that $\d_{x_0}$ is an extreme point of $(\BT)^*_1$.
Let $\chi_1, \chi_2 \in (\BT)_1^*$ be such that $\d_{x_0} = (\chi_1 + \chi_2)/2$.
Then $\chi_1(\t{\unit}) + \chi_2(\t{\unit}) = 2\d_{x_0}(\t{\unit}) = 2$ by \eqref{tf}.
Since $\chi_j \in (\BT)_1^*$, we get $|\chi_j(\t{\unit})| \leq 1$ and thus
$\chi_j(\t{\unit}) = 1 = \V{\chi_j}$ for $j = 1,2$.
Let $\tau_j$ be a representing measure for $\chi_j$, 
that is, $\D \chi_j(\t{f}) = \int_{\Kvs} \t{f} \,d\tau_j$ for $\t{f} \in \BT$
with $\V{\chi_j} = \V{\tau_j}$, the total variation of $\tau_j$, for $j=1,2$.
We see that $\tau_j$ is a probability measure on $\Kvs$ (cf. \cite[p. 81]{bro}).
Then $(\tau_1 + \tau_2)/2$ is a representing measure for $\d_{x_0}$.
It follows from Lemma~\ref{lem2.9} that $(\tau_1 + \tau_2)/2$ is
the Dirac measure $\gam_{x_0}$ concentrated at $x_0$. 
Because $\tau_j$ is a positive measure, $\tau_j(E) = 0$ for each
Borel set $E$ with $x_0 \not\in E$. Hence $\tau_j = \gam_{x_0}$ for $j=1,2$,
and consequently $\chi_1 = \chi_2$. Therefore, 
$\d_{x_0}$ is an extreme point of $(\BT)_1^*$, and thus $x_0 \in \chb$.
\end{proof}

\begin{rem}\label{rem2.2}
By Lemma~\ref{lem2.10}, we see that
$\KC=\Ch{B_C}$ and $\Ks=\Ch{B_\sigma}$.
The authors are not sure whether or not $\KS=\Ch{B_\Sigma}$ holds.
But, the following result on $\KS$ is enough for our proof.
\end{rem}


\begin{lem}\label{lem2.11.0}
Define $\KS^\circ = \set{(z,\eta,w) \in\KS : \eta\not\in \M_z}$.
Then the set $\KS^\circ$ is dense in $\KS=\DS\times \T$.
\end{lem}

\begin{proof}
Let $x_0 = (z_0, \eta_0, w_0) \in\KS$ and $O$ an
open neighborhood of $x_0$ in $\KS$.
We show that $O \cap \XS \neq \emptyset$.
To this end, we need to consider the case when $x_0 \not\in\XS$,
and hence $\eta_0 \in \M_{z_0}$.
By the definition of the product topology, we can choose open sets
$Z,W\subset \T$ and $V \subset \sa$ such that
$(z_0,\eta_0,w_0)\in Z\times V \times W\subset O$.
Choose $z_1 \in Z$ with $z_0 \neq z_1$, and then we obtain
$(z_1, \eta_0, w_0) \in O$ and
$\wh{\id}(\eta_0) = z_0 \neq z_1$. Hence $\eta_0 \not\in  \M_{z_1}$,
and therefore, $(z_1, \eta_0, w_0) \in O \cap\XS$.
This shows that $\XS$ is dense in $\KS$.
\end{proof}

We now define an isometry, $S$, on $\BT$ and investigate it.
To do this, we introduce a kind of adjoint operator, $S_*$,
on the dual space $(\BT)^*$ of $\BT$.
In the rest of this section, we shall prove that the map $S_*$
preserves extreme points of $(\BT)^*_1$ in some sense.

Let $(\Kvs, \VSS{\cdot})$ be one of $(\KS, \VS{\cdot})$ or
$(\KC, \VC{\cdot})$ or $(\Ks, \Vs{\cdot})$.
For a surjective, not necessarily linear, isometry $T$ from $(\Ao, \VSS{\cdot})$
onto itself, define the mapping $T_0$ on $(\Ao, \VSS{\cdot})$ by
\begin{equation}\label{T_0}
T_0 = T - T(0).
\end{equation}
By the Mazur-Ulam theorem \cite{maz, vai}, $T_0$ is a 
surjective, {\it real linear} isometry from $(\Ao, \VSS{\cdot})$ onto itself. 
We define the map $U \colon (\Ao, \VSS{\cdot}) \to (\BT, \DV{\cdot}_{\Kvs})$ by
$$
U(f) = \tf 
\qq (f \in \Ao),
$$
where $\tf$ is defined as in \eqref{tf}. By \eqref{norm} and \eqref{prop2.1.1},
$U$ is a surjective complex linear isometry.
Denote $UT_0U^{-1}$ by $S$; the mapping $S \colon \BT \to \BT$ is
a well-defined surjective {\it real linear} isometry,
since $U \colon (\Ao, \VSS{\cdot}) \to (\BT, \DV{\cdot}_{\Kvs})$ is a surjective
complex linear isometry. 
$$
\begin{CD} 
\Ao @>{T_0}>> \Ao \\ 
@V{U}VV
@VV{U}V \\ 
\BT @>>{S}> \BT
\end{CD} 
$$ 
By the definition of the mapping $S$, we derive
\begin{equation}\label{S}
S(\t{f}) = \widetilde{T_0(f)} \qq (\t{f} \in \BT).
\end{equation}
We define the mapping $S_* \colon (\BT)^* \to (\BT)^*$ by
\begin{equation}\label{S_*}
S_*(\chi)(\tf) = \Re[\chi(S(\t{f}))] - i \, \Re [\chi(S(i\t{f}))]
\end{equation}
for $\chi \in (\BT)^*$ and $\tf \in \BT$.
Here we note that $S_*$ coincides with the usual adjoint operator,
provided that $S$ is complex linear.
The mapping $S_*$ is a surjective real linear 
isometry with respect to the operator norm on $(\BT)^*$
(see, for example, \cite[Proposition~5.17]{rud}). 
This shows that $S_*$ preserves extreme points, that is, 
$S_*(\ext{(\BT)_1^*}) = \ext{(\BT)_1^*}$.

We define the topological subspace $\B$ of $(\BT)^*_1$ with the weak*-topology by
$$
\B=\set{\la\d_x: \la\in\T, \, x\in\Kvs}.
$$
Now we consider the case when $\Kvs\in\set{\KC,\Ks}$.
Then $\Kvs=\Ch{\BT}$ by Remark~\ref{rem2.2}.
Applying the Arens-Kelley theorem to the above, we obtain
$$
\ext{(\BT)_1^*}=\set{\la\d_x:\la\in\T,\ x\in\Ch{\BT}}=\B.
$$
Since $S_*$ preserves extreme points, we get $S_*(\B)=\B$.
In the rest of this section, we shall prove that the equality is true
even if $\Kvs=\KS$.

For $\Kvs\in\set{\KS,\KC,\Ks}$, we define the mapping $\h\colon\T\times\Kvs\to\B$ by
\begin{equation}\label{h}
\h(\la,x)=\la\d_x
\qq ((\la,x)\in\T\times\Kvs).
\end{equation}
We shall prove that $\h$ is a homeomorphism.
The property on $\h$ is quite important to show that $S_*$ preserves the set $\B$.
The ideas of the following two lemmas are due to Professor Kazuhiro Kawamura.
The authors are grateful for his suggestions and comments.


\begin{lem}\label{lem2.12}
Let $\T \times \Kvs$ be the compact Hausdorff space
with the product topology.
Then the mapping $\h\colon\T\times\Kvs\to\B$
is a homeomorphism.
In particular, $\h(\T \times \Ch{\BT}) = \ext{(\BT)_1^*}$.
\end{lem}

\begin{proof}
The map $\h$ is surjective by definition.
Because $\BT$ contains the constant function $\t{\unit}$
and separates the points of $\Kvs$
(see Propositions~\ref{prop2.2} and \ref{prop2.3}),
we see that $\h$ is injective.
By the definition of the weak*-topology, we observe that $\h$ is a continuous
map from the compact space $\T \times\Kvs$ onto the Hausdorff space $\B$.
Hence, $\h\colon\T\times\Kvs\to\B$ is a homeomorphism.
By the Arens-Kelley theorem, $\ext{(\BT)_1^*}=\set{\la\d_x:\la\in\T,\,x\in\Ch{\BT}}$,
and thus $\ext{(\BT)_1^*}=\h(\T\times\Ch{\BT})$.
\end{proof}

The next lemma is a generalization of the well-known fact that the adjoint
operator of a surjective complex linear isometry on $(\Ao,\VSS{\cdot})$
preserves the set $\ext{(\BT)_1^*}$ of all extreme points of $(\BT)_1^*$.
In fact, if $\Kvs=\KC$ or $\Kvs=\Ks$, then $\B=\ext{(\BT)_1^*}$, and thus
$S_*(\B)=\B$, as mentioned above.
We shall prove that $S_*(\B)=\B$ holds even if $\Kvs=\KS$


\begin{lem}\label{lem2.13}
The mapping $S_*$ preserves $\B$, that is,
$S_*(\B)=\B$.
\end{lem}

\begin{proof}
We need to consider the case when $\Kvs=\KS$, and then $\B=\BTS$.
Let $\h$ be the homeomorphism as in \eqref{h}.
Lemma~\ref{lem2.12} shows that
$$
\h(\T \times \Ch{\BS}) = \ext{(\BS)_1^*} = S_*(\ext{(\BS)_1^*}).
$$
By Lemma~\ref{lem2.10}, $\KS^\circ\subset\Ch{\BS}\subset\KS$,
and hence we observe that
\begin{align*}
S_*(\h(\T \times\KS^\circ))
&\subset
S_*(\h(\T\times\Ch{\BS})) = S_*(\ext{(\BS)_1^*}) \\
&=
\ext{(\BS)_1^*} = \h(\T \times \Ch{\BS}) \subset \h(\T \times \KS).
\end{align*}
By the definition of the map $\h$, we get $\h(\T \times \KS) = \BTS$,
and hence $S_*(\h(\T \times \KS^\circ)) \subset \BTS$.
We denote by $\clk{M}{L}$ the closure of a subset $L$ in a topological space $M$.
By Lemma~\ref{lem2.11.0}, $\KS^\circ$ is dense in
$\KS$, and thus $\clk{\KS}{\KS^\circ}=\KS$.
Since the map $\h$ is a homeomorphism, we obtain
$$
\BTS = \h(\T \times \KS)
= \h(\T \times \clk{\KS}{\KS^\circ}) = \clk{\BTS}{\h(\T \times \KS^\circ)},
$$
and thus $S_*(\BTS) = S_*(\clk{\BTS}{\h(\T \times \KS^\circ)})$.
Because $S_* \colon (\BS)^* \to (\BS)^*$ is a surjective isometry
with the operator norm, it is bijective.
By the definition of $S_*$, it is continuous with respect to the weak*-topology
on $(\BS)^*$. In the same way, we can show that $S_*^{-1}$ is continuous on
$(\BS)^*$ as well. Hence $S_* \colon (\BS)^* \to (\BS)^*$ is
a homeomorphism with respect to the weak*-topology.
Since $S_*(\h(\T \times \KS^\circ)) \subset \BTS$, we obtain
\[
S_*(\BTS) 
=S_*(\clk{\BTS}{\h(\T \times \KS^\circ)})
= \clk{\BTS}{S_*(\h(\T \times \KS^\circ))} \subset \BTS,
\]
which implies $S_*(\BTS) \subset \BTS$.
By the same arguments, applied to $S_*^{-1}$ instead of $S_*$, we have
$S_*^{-1}(\BTS) \subset \BTS$.
This shows that $S_*(\BTS) = \BTS$.
\end{proof}


\section{Characterization of the induced map $S$}\label{sect3}

In this section, we use the same notation as in the previous section.
We assume that $S \colon \BT \to \BT$ is a surjective real linear isometry, which
satisfies \eqref{S}, and $S_* \colon(\BT)^* \to (\BT)^*$ is a surjective real linear isometry
defined by \eqref{S_*}.
The mapping $\h\colon\T\times\Kvs\to\B$ is a
homeomorphism as in \eqref{h}.

We shall analyze the isometry $S_*$ on $(\BT)^*$, which leads to the characterization
of the map $S=UT_0U^{-1}$:
The idea of the arguments is a generalization of a well-known proof of the
Banach-Stone theorem.
In fact, if $S$ is complex linear, then the fact that the adjoint operator $S^*$
preserves extreme points implies that $S$ is a weighted composition operator.
Our argument below is a kind of real linear version of the above.


\begin{defn}\label{defn1}
Let $p_1 \colon \T \times \Kvs\to \T$ and $p_2 \colon \T \times\Kvs\to\Kvs$
be the natural projections
from $\T \times\Kvs$ onto the first and second coordinates, respectively.
We define maps $\al \colon \T \times \Kvs \to \T$
and $\Phi \colon \T \times\Kvs\to\Kvs$ by
$\al = p_1 \circ \h^{-1} \circ S_* \circ \h$ and
$\Phi = p_2 \circ \h^{-1} \circ S_* \circ \h$.
$$ 
\begin{CD}
\T \times \Kvs @>{}>> \T \times \Kvs \\
@V{\h}VV
@VV{\h}V \\ 
\B @>>{S_*}> \B
\end{CD} 
$$
Recall that $S_*$ is a real linear isometry with $S_*(\B) = \B$
by Lemma~\ref{lem2.13}, and thus
it is bijective. By the definition of $S_*$, it is continuous with respect to
the weak*-topology. Therefore, $S_*|_{\B}\colon\B\to\B$ is a homeomorphism
with respect to the relative weak*-topology.
Hence $\al$ and $\Phi$ are both well-defined, surjective continuous maps.
\end{defn}

We get $(\h^{-1} \circ S_* \circ \h)(\la,x)
= (\al(\la,x), \Phi(\la,x))$ for $(\la,x) \in \T \times \Kvs$
by the definition of $\al$ and $\Phi$.
Hence $(S_* \circ \h)(\la,x) = \h(\al(\la,x), \Phi(\la,x))$, which implies
$S_*(\la \d_x) = \al(\la,x) \dphi{\la,x}$
for every $(\la,x) \in \T \times \Kvs$. Set, for each $\la \in \T$,
$\al_\la(x) = \al(\la,x)$ and $\Phi_\la(x) = \Phi(\la,x)$ for $x \in \Kvs$.
Then $\al_\la$ and $\Phi_\la$ satisfy
$$
S_*(\la \d_x) = \al_\la(x) \d_{\Phi_\la(x)}
$$
for every $\la\in\T$ and $x\in\Kvs$.

We first investigate the map $\al$.
In fact, since $S_*$ is a real linear isometry, we can prove that
$\al(\cdot,x)$ is a real linear isometry on $\T$ for each $x\in\Kvs$
as well.


\begin{lem}\label{lem3.1}
For each $x \in \Kvs$, $\al_i(x) = i\al_1(x)$ or $\al_i(x) = -i\al_1(x)$.
\end{lem}


\begin{proof}
Let $x \in \Kvs$, and we set $\la_0 = (1+i) /\sqrt{2} \in \T$.
By the real linearity of $S_*$, we obtain
\[
\sqrt{2}\, \al_{\la_0}(x) \dphi{\la_0,x}
=S_*(\sqrt{2}\,\la_0\d_x) 
= S_*(\d_x) + S_*(i\d_x)
=\al_1(x)\dphi{1,x} + \al_i(x)\dphi{i,x}.
\]
Evaluating these equalities at $\tio \in \BT$, we get
$\sqrt{2}\, \al_{\la_0}(x) = \al_1(x) + \al_i(x)$. 
Since $|\al_\la(x)| = 1$ for all $\la \in \T$,
we have $\sqrt{2} = |\al_1(x) + \al_i(x)| = |1+\al_i(x)\ov{\al_1(x)}|$.
Then we see that $\al_i(x)\ov{\al_1(x)} = i$ or $\al_i(x)\ov{\al_1(x)} = -i$.
We thus conclude $\al_i(x) = i \al_1(x)$ or $\al_i(x) = -i\al_1(x)$.
\end{proof}

Now we can describe a connection between $\al_1$ and $\al_i$.


\begin{lem}\label{lem3.2}
There exists a continuous function $\ez \colon \Kvs \to \set{\pm 1}$ such that
$S_*(i\d_x) = i\ez(x)\alo(x)\dphi{i,x}$ for every $x \in \Kvs$.
\end{lem}


\begin{proof}
For each $x \in \Kvs$, we obtain $\al_i(x) = i\alo(x)$ or $\al_i(x) = -i\alo(x)$ by Lemma~\ref{lem3.1}.
We define $E_+$ and $E_-$ by
$$
E_+ = \set{x \in \Kvs : \al_i(x) = i\alo(x)}
\q\mbox{and}\q 
E_- = \set{x \in \Kvs : \al_i(x) = -i\alo(x)}. 
$$
Then we have $\Kvs = E_+ \cup E_-$ and $E_+ \cap E_- = \emptyset$. 
By the continuity of the functions $\al_1 = \al(1,\cdot)$ and $\al_i = \al(i,\cdot)$ on $\Kvs$,
we see that $E_+$ and $E_-$ are both closed subsets of $\Kvs$. 
Hence, the function $\ez \colon \Kvs \to \set{\pm 1}$, defined by
$$
\ez(x) = 
\begin{cases}
\ph{-}1 & x \in E_+ \\[1mm]
-1 & x \in E_-
\end{cases},
$$
is continuous on $\Kvs$.
Then we obtain $\al_i(x) = i\ez(x)\alo(x)$ for every $x \in \Kvs$. 
This shows $S_*(i\d_x) = i\ez(x)\alo(x)\d_{\Phi(i,x)}$ for all $x \in \Kvs$.
\end{proof}

We are ready to prove that $\al(\cdot,x)$ is a real linear isometry on
$\T$ for each $x\in\Kvs$; more explicitly, we have the following identity,
which says that $\Phi_\la$ depends on $\Phi_1$ and $\Phi_i$.


\begin{lem}\label{lem3.3}
Let $\ez$ be the continuous function as in Lemma~\ref{lem3.2}.
For each $\la = r+it \in \T$ with $r,t \in \R$ and $x \in \Kvs$,
\begin{equation}\label{lem3.3.1}
\la^{\ez(x)}\tf(\Phi_\la(x)) = r\tf(\Phi_1(x)) + it\ez(x)\tf(\Phi_i(x))
\end{equation}
for all $\tf \in \BT$.
\end{lem}


\begin{proof}
Let $\la = r+it \in \T$ with $r,t \in \R$ and $x \in \Kvs$.
Recall that $S_*(\d_x) = \alo(x)\dphi{1,x}$, and
$S_*(i\d_x) = i\ez(x) \alo(x)\dphi{i,x}$ by Lemma~\ref{lem3.2}.
Because $S_*$ is real linear,
\[
\al_\la(x)\dphi{\la,x}
=S_*(\la \d_x) 
= rS_*(\d_x) + tS_*(i\d_x) 
=r\alo(x)\dphi{1,x} + it\ez(x)\alo(x)\dphi{i,x}.
\]
The evaluation of these equalities at $\tio \in \BT$ shows that
$\al_\la(x) = \alo(x)(r+it\ez(x))$.
Since $r+its_0(x) = \la^{s_0(x)}$, then $\al_\la(x) = \la^{\ez(x)} \alo(x)$.
In particular, $\al(\cdot,x)$ is a real linear isometry on $\T$.
We thus obtain $\la^{\ez(x)}\dphi{\la,x} = r\dphi{1,x} + it\ez(x)\dphi{i,x}$,
which implies that
$\la^{\ez(x)}\tf(\Phi_\la(x)) = r\tf(\Phi_1(x)) + it\ez(x)\tf(\Phi_i(x))$
for all $\tf \in \BT$.
\end{proof}

We next analyze the map $\Phi$, which satisfies
$S_*(\la\d_x)=\al_\la(x)\d_{\Phi_\la(x)}$ for every
$\la\in\T$ and $x\in\Kvs$.
Roughly speaking, we prove that $\Phi_\la(x)$ is independent
from $\la\in\T$.


\begin{defn}\label{defn2}
Recall that $\Kvs=\Dvs\times\T$, where
$$
\DS=\T\times\sa,\q
\DC=\set{(\eta,\eta):\eta\in\sa},
\q\mbox{and}\q
\Ds=\set{a}\times\sa.
$$
We denote by $q_j$ the projection from $\Kvs$
onto the $j$-th coordinate of $\Kvs$ for $j=1,2,3$:
If $\Kvs=\KC$, then we define $q_j(\eta,\eta,w)=\eta$
for $j=1,2$ and $(\eta,\eta,w)\in\KC$.
Let $\Phi\colon \T \times \Kvs \to \Kvs$ be the map as in Definition~\ref{defn1}.
Define
$\phi \colon \T \times \Kvs \to q_1(\Kvs)$, $\psi \colon \T \times \Kvs \to \sa$
and $\varphi \colon \T \times \Kvs \to \T$ by 
$$
\phi = q_1 \circ \Phi, \qq \psi = q_2 \circ \Phi \qq\mbox{and}\qq
\varphi = q_3 \circ \Phi,
$$
respectively.
For each $\la \in \T$, we also define $\phila$, $\psila$ and $\varphila$
by $\phila(x) = \phi(\la,x)$, 
$\psi_\la(x) = \psi(\la,x)$ and $\varphi_\la(x) = \varphi(\la, x)$ for all $x \in \Kvs$.
\end{defn}

By the definition of $\phi$, $\psi$ and $\varphi$, we get
$\Phi_\la(x) = (\phila(x), \psila(x), \varphila(x))$
for every $\la\in \T$ and $x\in\Kvs$; if
$\Kvs=\KS$ then $\phila(x)\in\T$, if
$\Kvs=\KC$ then $\phila(x)=\psila(x)\in\sa$,
and if $\Kvs=\Ks$ then $\phila(x)=a$.
By \eqref{tf}, 
\begin{equation}\label{defn2.1}
\tf(\Phi_\la(x)) = \wh{f}(\phila(x)) + \wh{f'}(\psila(x))\varphila(x)
\end{equation}
for all $f \in \Ao$ and $(\la,x) \in \T \times \Kvs$.
Note that $\phi$, $\psi$ and $\varphi$ are surjective and continuous,
since so is $\Phi$ (see Definition~\ref{defn1}).

\begin{rem}\label{rem3.1}
We notice that $\Phi_\mu(x)=\Phi_{-\mu}(x)$
for $\mu\in\set{1,i}$ and $x\in\Kvs$: 
In fact, equality \eqref{lem3.3.1} shows that
\[
-\tf(\Phi_{-1}(x))=-\tf(\Phi_1(x))
\q\mbox{and}\q
-i\ez(x)\tf(\Phi_{-i}(x))=-i\ez(x)\tf(\Phi_i(x))
\]
for all $\tf\in\BT$ and $x\in\Kvs$.
Since $\BT$ separates the points of $\Kvs$ by Propositions~\ref{prop2.2}
and \ref{prop2.3}, we obtain $\Phi_1(x)=\Phi_{-1}(x)$ and
$\Phi_i(x)=\Phi_{-i}(x)$ for all $x\in\Kvs$.
Consequently, $\phio=\phi_{-1}$, $\psio=\psi_{-1}$ and $\varphio=\varphi_{-1}$.
\end{rem}

We shall prove that $\phila(x)$, $\psila(x)$ and $\varphila(x)$ are independent
from $\la\in\T$, which is a consequence of \eqref{lem3.3.1} and \eqref{defn2.1}.
We need the following elementary properties of $\Ao$ to demonstrate it.
In fact, we shall prove that $\Ao$ separates the points of $\Db$ and $\sa$
in the next two lemmas.


\begin{prop}\label{prop3.4}
Let $z_0,z_1,z_2\in\Db$ with $z_0\not\in\set{z_1,z_2}$, and let
$\eta_1,\eta_2,\eta_3\in\sa$.
There exists $f_0\in\Ao$ such that
\[
\wh{f_0}(z_0) = 1, \q \wh{f_0}(z_1) = 0 = \wh{f_0}(z_2)
\q\mbox{and}\q \wh{f_0'}(\eta_j) = 0 \qq (j=1,2,3).
\]
\end{prop}


\begin{proof}
We set $z_3=\wh{\id}(\eta_1)$, $z_4=\wh{\id}(\eta_2)$ and
$z_5=\wh{\id}(\eta_3)$.
Then $z_j\in\T$ for $j=3,4,5$.

{\bf Case 1.}
Suppose that $z_0 \not\in \set{z_3, z_4, z_5}$.
We set
$$
g(z)=(z-z_1)(z-z_2)(z-z_3)^2(z-z_4)^2(z-z_5)^2\in\Ao,
$$
and then
$\wh{g}(z_0) \neq 0$, $\wh{g}(z_1)=0=\wh{g}(z_2)$ and $\wh{g'}(z_j) = 0$ for $j=3,4,5$.
If we set $f_1 = g/\wh{g}(z_0) \in \Ao$, then we obtain
$$
\wh{f_0}(z_0) = 1, \q \wh{f_0}(z_1) = 0 = \wh{f_0}(z_2)
\q\mbox{and}\q \wh{f_0'}(\eta_j) = 0 \qq (j=1,2,3).
$$

{\bf Case 2.}
If $z_0 \in \set{z_3, z_4, z_5}$, then we define $J=\set{j\in\set{3,4,5}:z_j\neq z_0}$
and $J_0=\set{1,2}\cup J$. Let $j\in J_0$, and set
$$
g_j(z)=(2z+z_j-3z_0)(z-z_j)^2\in\Ao.
$$
Here, we notice that $z_0\neq z_j$.
We see that $\wh{g_j}(z_0)\neq 0$ and $\wh{g_j'}(z_0)=0=\wh{g_j}(z_j)=\wh{g_j'}(z_j)$.
The function $f_j=g_j/\wh{g_j}(z_0)\in\Ao$ satisfies $\wh{f_j}(z_0) = 1$
and $\wh{f_j'}(z_0) = 0 = \wh{f_j}(z_j) = \wh{f_j'}(z_j)$.
We set $f = \Pi_{k\in J_0}f_k\in \Ao$, and then
$\wh{f}(z_0) = 1$ and $\wh{f}(z_j) = 0$. In particular, $\wh{f}(z_1) = 0 = \wh{f}(z_2)$.
We note that
$f' = f_j'\,\Pi_{k\in J_0\setminus\set{j}}f_k+f_j\,(\Pi_{k\in J_0\setminus\set{j}}f_k)'$
by the multiplication law.
Since $\wh{f_j}(z_j)=0=\wh{f_j'}(z_j)$, we obtain $\wh{f'}(z_j)=0$ for each $j\in J_0$.
If $l\not\in J_0$, then $z_l=z_0$.
Set $J_0=\set{k_1,\dots,k_m}$, a finite set. Since $\wh{f_{k_1}}(z_0)=1$ and
$\wh{f_{k_1}'}(z_0)=0$, the multiplication law shows that
\[
\wh{f'}(z_l)
=\wh{f'}(z_0)
=\wh{f_{k_1}}(z_0)\left\{\left(\Pi_{k\in J_0\setminus\set{k_1}}f_k\right)'\right\}\wh{\mathstrut}\,(z_0)
=\left\{\left(\Pi_{k\in J_0\setminus\set{k_1}}f_k\right)'\right\}\wh{\mathstrut}\,(z_0).
\]
By the same reasonings, we have
$\wh{f'}(z_l)
=\left\{\left(\Pi_{k\in J_0\setminus\set{k_1,\cdots,k_{m-1}}}f_k\right)'\right\}\wh{\mathstrut}\,(z_0)
=\wh{f_{k_m}'}(z_0)=0.$
Hence $f$ is a desired function.
\end{proof}


\begin{prop}\label{prop3.5}
For each $\eta_0,\eta_1,\eta_2\in\sa$
with $\eta_0\not\in\set{\eta_1,\eta_2}$,
there exists $f_0\in\Ao$ and $j_0\in\set{0,1,2}$ such that
\[
\wh{f_0'}(\eta_{j_0}) = 1, \q \wh{f_0'}(\eta_k)=0 \qq(k\in\set{0,1,2}\setminus\set{j_0})
\q\mbox{and}\q \wh{f_0}(\eta_j) = 0 \qq (j\in\set{0,1,2}).
\]
\end{prop}


\begin{proof}
Set $z_j=\wh{\id}(\eta_j)$ for $j=0,1,2$.

{\bf Case 1.}
Suppose that $z_0=z_1=z_2$.
Since $\eta_0\not\in\set{\eta_1,\eta_2}$, there exists $v_1\in\H$
such that $\wh{v_1}(\eta_0)=1$ and $\wh{v_1}(\eta_1)=0=\wh{v_1}(\eta_2)$.
We set $g_1=I_0(v_1)\in\Ao$, and then $\wh{g_1'}(\eta_0)=1$ and
$\wh{g_1'}(\eta_1)=0=\wh{g_1'}(\eta_2)$ (see \eqref{I}).
Because $\wh{g_1}$ is continuous on $\Di$,
we have that $\wh{g_1}(\xi)=\wh{g}(\eta)$ for each $\xi\in\T$ and $\eta\in\M_\xi$,
the fiber over $\xi$.
This implies that $\wh{g_1}(\eta_0)=\wh{g_1}(\eta_1)=\wh{g_1}(\eta_2)$,
since $z_0=z_1=z_2$.
Then the function $f_1=g_1-\wh{g_1}(\eta_0)\unit\in\Ao$ satisfies
\[
\wh{f_1'}(\eta_0) = 1, \q \wh{f_1'}(\eta_1) = 0 = \wh{f_1'}(\eta_2)
\q\mbox{and}\q \wh{f_1}(\eta_j) = 0 \qq (j\in\set{0,1,2}).
\]

{\bf Case 2.}
Suppose that $z_0=z_1\neq z_2$.
Setting $g_2(z)=(z-z_0)^2(z-z_2)\in\Ao$,
we see that $\wh{g_2}(z_0)=0=\wh{g}(z_2)$,
$\wh{g_2'}(z_0)=0$ and $\wh{g_2'}(z_2)\neq 0$.
Then $f_2=g_2/\wh{g_2'}(z_2)\in\Ao$ satisfies
\[
\wh{f'_2}(\eta_2)=1, \q \wh{f_2'}(\eta_0) = 0 = \wh{f_2'}(\eta_1)
\q\mbox{and}\q \wh{f_2}(\eta_j) = 0 \qq (j\in\set{0,1,2}).
\]

{\bf Case 3.}
Suppose that $z_0=z_2\neq z_1$.
By a quite similar argument to Case 2, we can find
$f_3\in\Ao$ with
\[
\wh{f'_3}(\eta_1)=1, \q \wh{f_3'}(\eta_0) = 0 = \wh{f_3'}(\eta_2)
\q\mbox{and}\q \wh{f_3}(\eta_j) = 0 \qq (j\in\set{0,1,2}).
\]

{\bf Case 4.}
Now we suppose that $z_0\neq z_1$ and $z_0\neq z_2$. Set
$$
g_4(z)=(z-z_0)(z-z_1)^2(z-z_2)^2\in\Ao.
$$
We see that $\wh{g_4'}(z_0)\neq 0$, $\wh{g_4'}(z_1)=0=\wh{g_4'}(z_2)$ and
$\wh{g_4}(z_j)=0$ for $j=0,1,2$.
Set $f_4=g_4/\wh{g_4'}(z_0)\in\Ao$, and then
\[
\wh{f'_4}(\eta_0)=1, \q \wh{f_4'}(\eta_1) = 0 = \wh{f_4'}(\eta_2)
\q\mbox{and}\q \wh{f_4}(\eta_j) = 0 \qq (j\in\set{0,1,2}).
\]
The proof is complete.
\end{proof}

Before proving that $\phila(x)$ is independent of $\la\in\T$,
we prepare one lemma, which is an essential part of it.

\begin{lem}\label{lem3.6}
Let $\Kvs\in\set{\KS,\KC,\Ks}$. 
Then $\phi_\la(x)\in\set{\phio(x),\phii(x)}$ for all $\la\in\T$ and $x\in\Kvs$.
\end{lem}


\begin{proof}
Note first that $\phi$ is a map from $\T\times\Kvs$ to $q_1(\Kvs)$.
If $\Kvs=\Ks$, then $q_1(\Kvs)=\set{a}$. 
Thus $\phila(x)=\phio(x)$ for all $\la\in\T$ and $x\in\Ks$.

Let $\Kvs\in\set{\KS,\KC}$ and $x_0 \in \Kvs$.
Suppose, on the contrary, that $\phi_{\la_0}(x_0) \not\in \set{\phio(x_0), \phii(x_0)}$
for some $\la_0 \in \T \setminus \set{\pm 1,\pm i}$
(see Remark~\ref{rem3.1}).

If $\Kvs=\KS$, then by Proposition~\ref{prop3.4}, there exists $f_0\in\Ao$ such that
\[
\wh{f_0}(\phi_{\la_0}(x_0))=1, \q \wh{f_0}(\phio(x_0))=0=\wh{f_0}(\phii(x_0))
\q\mbox{and}\q \wh{f_0'}(\psi_\mu(x_0)) = 0 \qq (\mu\in\set{\la_0,1,i}).
\]
By \eqref{defn2.1}, we have
$\t{f_0}(\Phi_{\la_0}(x_0)) = 1$ and $\t{f_0}(\Phi_1(x_0)) = 0 = \t{f_0}(\Phi_i(x_0))$.
Applying these equalities to \eqref{lem3.3.1} to get $\la_0^{\ez(x_0)} = 0$,
which contradicts $\la_0 \in \T$.

If $\Kvs=\KC$, then $\phi=\psi$ by definition.
Proposition~\ref{prop3.5} shows that there exist
$f_1\in\Ao$ and $\mu_0\in\set{\la_0,1,i}$ such that
\begin{multline*}
\wh{f_1'}(\phi_{\mu_0}(x_0)) = 1, \q \wh{f_1'}(\phi_\nu(x_0))=0
\qq(\nu\in\set{\la_0,1,i}\setminus\set{\mu_0})\\
\q\mbox{and}\q \wh{f_1}(\phi_\mu(x_0)) = 0 \qq(\mu\in\set{\la_0,1,i}).
\end{multline*}
Equality \eqref{defn2.1} shows that
$\t{f_1}(\Phi_{\mu_0}(x_0))=\varphi_{\mu_0}(x_0)$ and
$\t{f_1}(\Phi_\nu(x_0))=0$ for $\nu\in\set{\la_0,1,i}\setminus\set{\mu_0}$.
If $\mu_0=\la_0$, then equality \eqref{lem3.3.1} yields
$\la_0^{\ez(x_0)}\varphi_{\la_0}(x_0)=0$. This contradicts
$\la_0\varphi_{\la_0}(x_0)\in\T$. Hence $\mu_0\in\set{1,i}$.
Let $\la_0=r_0+it_0$ with $r_0,t_0\in\R$. If $\mu_0=1$, then
$0=r_0\varphio(x_0)$ by \eqref{lem3.3.1}, which implies $r_0=0$.
Thus $it_0=\la_0\in\T\setminus\set{\pm 1,\pm i}$, a contradiction.
If $\mu_0=i$, we will lead a similar contradiction by the same reasonings.

From the above arguments, we conclude that
$\phila(x_0) \in \set{\phio(x_0), \phii(x_0)}$ for all $\la \in \T$.
Since $x_0\in\Kvs$ was arbitrary, the proof is complete.
\end{proof}

Now we are in a position to show that $\phila$ does not depend
on $\la\in\T$.
We first prove it for $\Kvs=\KS$ and $\Kvs=\Ks$.


\begin{lem}\label{lem3.7}
Let $\Kvs\in\set{\KS,\Ks}$.
Then $\phio \colon \Kvs \to q_1(\Kvs)$ is a surjective and continuous function
with $\phio(x) = \phila(x)$ for all $x\in\Kvs$ and $\la \in \T$.
\end{lem}


\begin{proof}
If $\Kvs=\Ks$, then $\phi$ is a surjective and continuous function
from $\T\times\Ks$ onto $q_1(\Ks)=\set{a}$.
Hence $\phio\colon\Ks\to\set{a}$ is a surjective and continuous function
satisfying $\phio(x)=\phila(x)$ for all $x\in\Ks$ and $\la\in\T$.

Let $\Kvs=\KS$ and $x_0 \in \KS$. We shall prove that $\phio(x_0) = \phii(x_0)$.
Since $\phi$ is a continuous function from $\T\times\KS$ onto $q_1(\KS)=\T$,
the function $\phi(\cdot,x_0)\colon\T\to\T$, which maps $\la\in\T$ to
$\phi(\la,x_0)$, is continuous on $\T$ as well. This implies that the image 
$\phi(\T,x_0)$ of $\T$ under the function $\phi(\cdot,x_0)$ is a connected subset
of $\T$. Because $\phi(\T,x_0)\in\set{\phio(x_0),\phii(x_0)}$ by Lemma~\ref{lem3.6},
we conclude that $\phio(x_0)=\phii(x_0)$. We thus obtain $\phila(x_0)=\phio(x_0)$
for all $\la\in\T$.

We show that $\phio\colon\KS\to\T$ is surjective.
Since $\phi$ is surjective, for each $\xi\in\T$
there exists $(\la_1, x_1) \in \T \times \KS$ such that $\phi(\la_1,x_1) = \xi$.
Then we have that $\phio(x_1) = \phi_{\la_1}(x_1) = \phi(\la_1,x_1) = \xi$.
This shows that $\phio$ is surjective.
\end{proof}

We now prove that the above result holds for $\Kvs=\KC$.


\begin{lem}\label{lem3.8}
The map $\phio \colon \KC \to\sa$ is surjective and continuous.
In addition, $\phio(x) = \phila(x)$ holds for all $x\in\KC$ and $\la \in \T$.
\end{lem}


\begin{proof}
Note that $\phi=\psi$ on $\T\times\KC$ by definition.
Let $x_0 \in \KC$. We shall prove that $\phio(x_0) = \phii(x_0)$.
To this end, suppose not, and then
$\phio(x_0) \neq \phii(x_0)$.
Let $\la_0 = (1+i)/\sqrt{2} \in \T$, and set
$\eta_0=\phi_{\la_0}(x_0)$, $\eta_1=\phio(x_0)$
and $\eta_2=\phii(x_0)$.
We see that $\eta_0\in\set{\eta_1,\eta_2}$
by Lemma~\ref{lem3.6}.
If $\eta_0=\eta_1\neq\eta_2$, Proposition~\ref{prop3.5} shows that
there exist $f_0\in\Ao$ and $j_0\in\set{0,1,2}$ such that
\[
\wh{f_0'}(\eta_{j_0})=1, \q \wh{f_0'}(\eta_k)=0 \qq(k\in\set{0,1,2}\setminus\set{j_0})
\qbox{and}\q \wh{f_0}(\eta_j)=0\qq (j\in\set{0,1,2}).
\]
Substituting these equalities into \eqref{defn2.1}, we get
$\t{f_0}(\Phi_{\mu_0}(x_0))=\varphi_{\mu_0}(x_0)$ for some $\mu_0\in\set{\la_0,1,i}$ and
$\t{f_0}(\Phi_\nu(x_0))=0$ for $\nu\in\set{\la_0,1,i}\setminus\set{\mu_0}$.
We can lead a contradiction by a similar argument to Proof of Lemma~\ref{lem3.6}.
If $\eta_0=\eta_2\neq\eta_1$, we will arrive at a contradiction
by the same reasonings.
Thus, we have $\phio(x_0) = \phii(x_0)$, and therefore,
$\phila(x_0)=\phio(x_0)$ for all $\la\in\T$.
Since $x_0\in\KC$ is arbitrarily chosen, we conclude that
$\phio(x) = \phila(x)$ for all $x\in\KC$ and $\la \in \T$.

We show that $\phio$ is surjective.
For each $\eta\in q_1(\KC)$, we can choose $(\la_1,x_1)\in\T\times\KC$
such that $\phi(\la_1,x_1)=\eta$, since $\phi$ is surjective.
Then $\phio(x_1) = \phi_{\la_1}(x_1) = \phi(\la_1,x_1)=\eta$.
Hence, $\phio$ is surjective.
\end{proof}

We shall prove that $\psila(x)$ is independent of $\la\in\T$.
The proof below is quite similar to that of Lemmas~\ref{lem3.7}.


\begin{lem}\label{lem3.9}
Let $\Kvs\in\set{\KS,\KC,\Ks}$.
Then $\psio \colon \Kvs \to \sa$ is a surjective and continuous map
with $\psio(x) = \psila(x)$ for all $x\in\Kvs$ and $\la \in \T$.
\end{lem}


\begin{proof}
If $\Kvs=\KC$, then Lemma~\ref{lem3.8} yields the conclusion,
since $\phi=\psi$ for $\KC$.
We will consider the case when $\Kvs\in\set{\KS,\Ks}$.
Let $x_0 \in \Kvs$.
By Lemma~\ref{lem3.7}, $\Phi_\la(x_0) = (\phio(x_0), \psila(x_0), \varphila(x_0))$
for $\la \in \T$.
Equality \eqref{defn2.1} is reduced to
\begin{equation}\label{lem3.5.1}
\tf(\Phi_\la(x_0)) = \wh{f}(\phio(x_0)) + \wh{f'}(\psila(x_0)) \varphila(x_0)
\end{equation}
for all $f \in \Ao$ and $\la \in \T$.

First, we show that $\psila(x_0) \in \set{\psio(x_0), \psii(x_0)}$ for all $\la \in \T$.
Suppose, on the contrary, that $\psi_{\la_0}(x_0) \not\in \set{\psio(x_0), \psii(x_0)}$
for some $\la_0 \in \T \setminus \set{1,i}$.
Set $\eta_0=\psi_{\la_0}(x_0)$, $\eta_1=\psio(x_0)$ and $\eta_2=\psii(x_0)$.
Then $\eta_j\in\sa$ for $j=0,1,2$ and $\eta_0\not\in\set{\eta_1,\eta_2}$.
There exists $v_0\in\H$ such that
$\wh{v_0}(\eta_0)=1$ and $\wh{v_0}(\eta_1)=0=\wh{v_0}(\eta_2)$.
The function $f_0\in\Ao$, defined by $f_0=I_0(v_0)-\wh{I_0(v_0)}(\phio(x_0))\unit$,
satisfies that $\wh{f_0}(\phio(x_0))=0$, $\wh{f_0'}(\eta_0)=1$ and
$\wh{f_0'}(\eta_1)=0=\wh{f_0'}(\eta_2)$.
By \eqref{lem3.5.1}, $\t{f_0}(\Phi_{\la_0}(x_0)) = \varphi_{\la_0}(x_0)$ and
$\t{f_0}(\Phi_1(x_0)) = 0 = \t{f_0}(\Phi_i(x_0))$.
If we substitute these equalities into \eqref{lem3.3.1}, we obtain 
$\la_0^{\ez(x_0)}\varphi_{\la_0}(x_0) = 0$, which contradicts $\la_0, \varphi_{\la_0}(x_0)\in\T$.
Consequently, $\psila(x_0)\in\set{\psio(x_0),\psii(x_0)}$ for all $\la \in \T$.

We next prove that $\psio(x_0) = \psii(x_0)$.
Suppose that $\psio(x_0) \neq \psii(x_0)$.
Let $\la_1 = (1+i)/\sqrt{2} \in \T$. Then $\psi_{\la_1}(x_0) \in \set{\psio(x_0), \psii(x_0)}$
as proved above.
If $\psi_{\la_1}(x_0) = \psio(x_0)\neq\psii(x_0)$, then we can choose
$f_1 \in \Ao$ so that 
$$
\wh{f_1}(\phio(x_0)) = 0 = \wh{f_1'}(\psio(x_0))
\qbox{and}\q \wh{f_1'}(\psii(x_0)) = 1
$$
by the same arguments as in the previous paragraph.
Applying these equalities to \eqref{lem3.5.1}, we have $\tf(\Phi_i(x_0)) = \varphii(x_0)$
and $\tf(\Phi_1(x_0)) = 0 = \tf(\Phi_{\la_1}(x_0))$, since $\psi_{\la_1}(x_0) = \psio(x_0)$.
By \eqref{lem3.3.1}, we get $0 = is_0(x_0)\varphi_i(x_0)$, which is impossible.
If $\psi_{\la_1}(x_0) = \psii(x_0)\neq\psio(x_0)$, then we will reach
a contradiction by a quite similar argument.
Therefore, we conclude that $\psio(x_0)=\psii(x_0)$.
Consequently $\psila(x) = \psio(x)$ for all $\la \in \T$ and $x \in \Kvs$.

Finally, since $\psi$ is surjective, for each $\eta \in \sa$ there exists 
$(\la, x) \in\T\times\Kvs$
such that $\psi(\la,x) = \eta$. We thus obtain
$\eta = \psila(x) = \psio(x)$, which shows the surjectivity of $\psio$.
\end{proof}

We investigate the function $\varphila$.
To be more explicit, we shall prove, for each $x\in\Kvs$, that
$\varphi_i(x)$ is $\varphi_1(x)$ or $-\varphi_1(x)$.


\begin{lem}\label{lem3.10}
Let $\Kvs\in\set{\KS,\KC,\Ks}$.
There exists a continuous function $\eo \colon \Kvs \to \set{\pm 1}$
such that $\varphii(x) = \eo(x)\varphio(x)$ for all $x \in \Kvs$.
\end{lem}


\begin{proof}
Fix an arbitrary $x_0\in\Kvs$. According to Lemmas~\ref{lem3.7}, \ref{lem3.8}
and \ref{lem3.9},
we can write
$\Phi_\la(x_0) = (\phio(x_0),\psio(x_0), \varphila(x_0))$ for all $\la \in \T$. 
Let $\la_0 = (1+i)/\sqrt{2} \in \T$ and
$f_0(z) = z - \phio(x_0)\unit\in \Ao$.
Then $\wh{f_0}(\phio(x_0)) = 0$ and $\wh{f_0'} = 1$ on $\sa$.
By \eqref{lem3.5.1}, $\t{f_0}(\Phi_\mu(x_0)) = \varphi_\mu(x_0)$ for $\mu = \la_0, 1, i$.
If we apply these equalities to \eqref{lem3.3.1}, then we obtain
$\sqrt{2}\,\la_0^{\ez(x_0)}\varphi_{\la_0}(x_0) = \varphio(x_0) + i\ez(x_0)\varphii(x_0)$.
As $\varphila(x_0) \in \T$ for all $\la \in \T$, we have
$$
\sqrt{2} = |\varphio(x_0)+i\ez(x_0) \varphii(x_0)|
= |1 + i\ez(x_0)\varphii(x_0)\ov{\varphio(x_0)}|.
$$
Then we get $i\ez(x_0)\varphii(x_0)\ov{\varphio(x_0)} = i$
or $i\ez(x_0)\varphii(x_0)\ov{\varphio(x_0)} = -i$.
Thus, for each $x \in \Kvs$, we obtain $\varphii(x) = \ez(x)\varphio(x)$ or
$\varphii(x) = -\ez(x) \varphio(x)$. 
By the continuity of $\varphio$, $\varphii$ and $\ez$, there exists a continuous function
$\eo \colon \Kvs \to \set{\pm 1}$ such that $\varphii(x) = \eo(x)\varphio(x)$ for all $x \in \Kvs$.
\end{proof}

In the rest of this paper, we denote $a + ibs$ by $[a+ib]^{s}$ for $a, b \in \R$ and
$s \in \set{\pm 1}$.
Thus, for each $\la \in \C$, $[\la]^{s} = \la$ if $s = 1$ and $[\la]^{s} = \ov{\la}$ if $s = -1$.
Therefore, $[\la\mu]^{s} = [\la]^{s} [\mu]^{s}$ for all $\la, \mu \in \C$. If, in addition, $\la \in \T$,
then $[\la]^{s} = \la^{s}$ for $s \in \set{\pm 1}$.

From the above arguments, we have a partial form of the isometry $S$ on $\BT$.
This is an important part of the isometry $T_0=U^{-1}SU$ on $\Ao$
(see \eqref{S}).


\begin{lem}\label{lem3.11}
Let $\Kvs\in\set{\KS,\KC,\Ks}$.
For each $f \in \Ao$ and $x \in \Kvs$,
\begin{equation}\label{lem3.11.1}
S(\tf)(x) = [\al_1(x)\wh{f}(\phio(x))]^{\ez(x)} 
+ [\al_1(x) \wh{f'}(\psio(x))\varphio(x)]^{\ez(x)\eo(x)}.
\end{equation}
\end{lem}


\begin{proof}
Let $f \in \Ao$ and $x \in \Kvs$.
By \eqref{S_*},
$\Re \set{S_*(\chi)(\tf)} = \Re \set{\chi(S(\tf))}$ for every $\chi \in (\BT)^*$.
Taking $\chi = \d_x$ and $\chi = i\d_x$ into the last equality, we get
$$
\Re \set{S_*(\d_x)(\tf)} = \Re \set{S(\tf)(x)}
\q\mbox{and}\q \Re \set{S_*(i\d_x)(\tf)}
= -\Im \set{S(\tf)(x)},
$$
respectively, where $\Im z$ is the imaginary part of a complex number $z$. Therefore,
\begin{equation}\label{lem3.11.2}
S(\tf)(x) = \Re \set{S_*(\d_x)(\tf)} - i \, \Re \set{S_*(i\d_x)(\tf)}.
\end{equation}
Recall that $S_*(\d_x) = \alo(x)\dphi{1,x}$, and $S_*(i\d_x) = i\ez(x) \alo(x)\dphi{i,x}$
by Lemma~\ref{lem3.2}.
Substitute these two equalities into \eqref{lem3.11.2} to obtain 
$$
S(\tf)(x) = \Re \set{\al_1(x)\tf(\Phi_1(x))}
+i \, \Im \set{\ez(x) \al_1(x)\tf(\Phi_i(x))}.
$$
Lemmas~\ref{lem3.7}, \ref{lem3.8}, \ref{lem3.9} and \ref{lem3.10} imply that
\[
\Phi_1(x) = (\phio(x), \psio(x), \varphio(x))
\q\mbox{and}\q
\Phi_i(x) = (\phio(x), \psi_1(x), \eo(x) \varphio(x)).
\]
Applying these two equalities to the above formula of $S(\tf)(x)$,
we assure from \eqref{tf} that 
\begin{align*}
S(\tf)(x)
&=
\Re\left\{\alo(x)\left(\wh{f}(\phio(x)) + \wh{f'}(\psio(x)) \varphio(x)\right)\right\} \\
&\qq
+ i \, \Im\left\{\ez(x)\alo(x)\left(\wh{f}(\phio(x))
+ \wh{f'}(\psio(x)) \eo(x) \varphio(x)\right)\right\} \\
&=
[\al_1(x) \wh{f}(\phio(x))]^{\ez(x)} + [\al_1(x) \wh{f'}(\psio(x)) \varphi_1(x)]^{\ez(x)\eo(x)}.
\end{align*}
This completes the proof.
\end{proof}


\section{Proof of the main results}\label{sect4}

We recall that $S(\tf) = \t{T_0(f)}$ for $f \in \Ao$ by \eqref{S}.
Applying \eqref{tf}, we can rewrite
equality \eqref{lem3.11.1} as
\begin{equation}\label{lem4.1.0}
\wh{T_0(f)}(z) + \wh{T_0(f)'}(\eta)w
= [\al_1(x)\wh{f}(\phio(x))]^{\ez(x)} 
+ [\al_1(x) \wh{f'}(\psio(x))\varphio(x)]^{\ez(x)\eo(x)}
\end{equation}
for all $f \in \Ao$ and $x = (z,\eta,w) \in \Kvs = \Dvs\times\T$.

In this section, we shall derive from \eqref{lem4.1.0}
that $\wh{T_0(f)}(z)= [\al_1(x)\wh{f}(\phio(x))]^{\ez(x)}$
and $\wh{T_0(f)'}(\eta)w=[\al_1(x) \wh{f'}(\psio(x))\varphio(x)]^{\ez(x)\eo(x)}$;
the main idea of its proof is to show that $\alo(x)$, $\phio(x)$ and $\ez(x)$ are
all independent of $w$ for $x=(z,\eta,w)\in\Kvs$.
Once we obtain it, then the first term of the right-hand side of \eqref{lem4.1.0} is
constant with respect to $w$, and therefore, we get 
$\wh{T_0(f)}(z)= [\al_1(x)\wh{f}(\phio(x))]^{\ez(x)}$.
For this end, we first show that $\phio(z,\eta,w)$ does not depend on $w\in\T$.

In Lemmas~\ref{lem4.1} and \ref{lem4.2}, we assume that $\Dvs$ is one of
$\DS=\T\times\sa$, $\DC=\set{(\eta,\eta):\eta\in\sa}$
and $\Ds=\set{a}\times\sa$ with $a\in\Di$.


\begin{lem}\label{lem4.1}
For each $(z,\eta)\in\Dvs$, the value $\phio(z,\eta,w)$ is independent of $w \in \T$.
\end{lem}


\begin{proof}
Take any $(z,\eta)\in\Dvs$ and $w_1, w_2\in\T$ with $w_1\neq w_2$.
We shall show that
$\phio(z,\eta,w)\in\set{\phio(z,\eta,w_1),\phio(z,\eta,w_2)}$
for all $w\in\T$.
If $\Kvs=\Ks$, then $\phio$ is a constant function with the value $a$
on $\Ks$. We need to consider the case when $\Kvs\in\set{\KS,\KC}$.
Suppose that there exists $w_0\in\T$ such that
$\phio(z,\eta,w_0)\not\in\set{\phio(z,\eta,w_1),\phio(z,\eta,w_2)}$.
We set $x_j=(z,\eta,w_j)$, $z_j=\phio(x_j)$ and $\eta_j=\psio(x_j)$
for $j=0,1,2$. Then $z_0\not\in\set{z_1,z_2}$ by hypothesis.
If $\Kvs=\KS$, then $z_j\in\T$ for $j=0,1,2$.
There exists $f_0\in\Ao$ such that
$$
\wh{f_0}(z_0)=1,\q\wh{f_0}(z_1)=0=\wh{f_0}(z_2)
\q\mbox{and}\q\wh{f_0'}(\eta_j)=0\qq(j=0,1,2),
$$
by Proposition~\ref{prop3.4}.
Substituting these equalities into \eqref{lem4.1.0}, we obtain
\begin{align*}
\wh{T_0(f_0)}(z) + \wh{T_0(f_0)'}(\eta)w_0
&=
[\alo(x_0)]^{\ez(x_0)},\\
\wh{T_0(f_0)}(z) + \wh{T_0(f_0)'}(\eta)w_1
&=
0 = \wh{T_0(f_0)}(z) + \wh{T_0(f_0)'}(\eta)w_2.
\end{align*}
Since $w_1\neq w_2$, the above equalities yield
$\wh{T_0(f_0)'}(\eta) = 0 = \wh{T_0(f_0)}(z)$.
Therefore, $[\alo(x_0)]^{\ez(x_0)} = 0$,
which is impossible, since $\alo(x_0) \in \T$.
Hence, $\phio(z,\eta,w)\in\set{\phio(z,\eta,w_1),\phio(z,\eta,w_2)}$
for all $w\in\T$.

If $\Kvs=\KC$, then $z_j=\eta_j\in\sa$ for $j=0,1,2$.
By Proposition~\ref{prop3.5}, there exist $f_1\in\Ao$ and $j_0\in\set{0,1,2}$
such that
\[
\wh{f_1'}(\eta_{j_0}) = 1, \q \wh{f_1'}(\eta_k)=0 \qq(k\in\set{0,1,2}\setminus\set{j_0})
\q\mbox{and}\q \wh{f_1}(\eta_j) = 0 \qq (j\in\set{0,1,2}).
\]
If we apply these equalities to \eqref{lem4.1.0}, we get
\begin{align*}
\wh{T_0(f_1)}(z) + \wh{T_0(f_1)'}(\eta)w_k
&=
0
&(k\in\set{0,1,2}\setminus\set{j_0}),\\
\wh{T_0(f_1)}(z) + \wh{T_0(f_1)'}(\eta)w_{j_0}
&=
[\alo(x_{j_0})\varphio(x_{j_0})]^{\ez(x_{j_0})\eo(x_{j_0})}.
\end{align*}
We notice that $w_0,w_1$ and $w_2$ are mutually distinct
by the choice of them. Then the above equalities show that
$\wh{T_0(f_1)'}(\eta) = 0 = \wh{T_0(f_1)}(z)$.
Therefore, $[\alo(x_{j_0})\varphio(x_{j_0})]^{\ez(x_{j_0})\eo(x_{j_0})}=0$,
which contradicts $\alo(x_{j_0}),\varphio(x_{j_0})\in \T$.
Hence, $\phio(z,\eta,w)\in\set{\phio(z,\eta,w_1),\phio(z,\eta,w_2)}$
for all $w\in\T$.

We now prove that the value $\phio(z, \eta, w)$ is independent of $w \in \T$.
Since $\phi$ is continuous, so is the function
$\phio(z,\eta,\cdot)\colon\T\to q_1(\Kvs)$,
which maps $w\in\T$ to $\phio(z,\eta,w)$.
Thus, the set $\phio(z, \eta, \T)$ is connected.
By the previous paragraph, we see that
$\phio(z,\eta,\T)\subset\set{\phio(z,\eta,w_1),\phio(z,\eta,w_2)}$.
We thus obtain $\phio(z,\eta,w_1)=\phio(z,\eta,w_2)$.
This implies that the value $\phio(z,\eta,w)$ is independent of $w \in \T$.
\end{proof}

We next prove that $\psio(z,\eta,w)$ does not depend on $w\in\T$;
proof of the result is quite similar to that of Lemma~\ref{lem4.1}.


\begin{lem}\label{lem4.2}
For each $(z,\eta)\in\Dvs$, the value $\psio(z,\eta,w)$ is independent of $w \in \T$.
\end{lem}


\begin{proof}
Fix arbitrary $(z,\eta)\in\Dvs$ and $w_1,w_2\in\T$ satisfying $w_1\neq w_2$.
We shall prove that $\psio(z,\eta,w)\in\set{\psio(z,\eta,w_1),\psio(z,\eta,w_2)}$
for all $w\in\T$.
Suppose, on the contrary, that there exists $w_0\in\T$ such that
$\psio(z,\eta,w_0)\not\in\set{\psio(z,\eta,w_1),\psio(z,\eta,w_2)}$.
We set $x_j = (z, \eta,w_j)$ and $\eta_j=\psio(x_j)$ for $j=0,1,2$.
Then $\eta_0\not\in\set{\eta_1,\eta_2}$ by the hypothesis.
Choose $v_0\in\H$ so that $\wh{v_0}(\eta_0)=1$ and
$\wh{v_0}(\eta_1)=0=\wh{v_0}(\eta_2)$.
Recall that $\phio(x_j)$ is independent of $j$ by Lemma~\ref{lem4.1}.
If we set $f_0=I_0(v_0)-\wh{I_0(v_0)}(\phio(x_0))\unit\in\Ao$, then we observe
$$
\wh{f_0}(\phio(x_0)) = 0, \q
\wh{f_0'}(\eta_0)=1
\qbox{and}\q 
\wh{f_0'}(\eta_1)=0=\wh{f_0'}(\eta_2).
$$
Equality \eqref{lem4.1.0} with the above implies that
\begin{align*}
\wh{T_0(f_0)}(z) + \wh{T_0(f_0)'}(\eta)w_0
&=
[\alo(x_0)\varphio(x_0)]^{\ez(x_0)\eo(x_0)}, \\
\wh{T_0(f_0)}(z) + \wh{T_0(f_0)'}(\eta)w_1
&=
0 = \wh{T_0(f_0)}(z) + \wh{T_0(f_0)'}(\eta)w_2.
\end{align*}
Because $w_1\neq w_2$, we assure from the above equalities that
$\wh{T_0(f_0)'}(\eta) = 0 = \wh{T_0(f_0)}(z)$.
Hence, $[\alo(x_0)\varphio(x_0)]^{\ez(x_0)\eo(x_0)} = 0$,
which contradicts $\alo(x_0)\varphio(x_0) \in \T$.
We thus conclude that $\psio(z,\eta,w)\in\set{\psio(z,\eta,w_1),\psio(z,\eta,w_2)}$
for all $w\in\T$.

We show that the value $\psio(z, \eta, w)$ is independent of $w \in \T$.
The set $\psio(z, \eta, \T) = \set{\psio(z, \eta, w) : w \in \T}$
is connected, since $\psio$ is continuous. 
By the last paragraph, we see that $\psio(z, \eta, w_1) = \psio(z, \eta, w_2)$.
Hence, the value $\psio(x, \eta, w)$ does not depend on $w \in \T$, as is claimed.
\end{proof}

Now we shall prove that both $\ez(z,\eta,w)$ and $\eo(z,\eta,w)$ are constant
functions with respect to $w\in\T$.


\begin{lem}\label{lem4.3}
\begin{enumerate}
\item
If $\Kvs\in\set{\KS,\Ks}$, then for $x = (z, \eta, w) \in \Kvs$,
the values $\ez(x)$ and $\eo(x)$ are independent
of $z\in q_1(\Kvs)$ and $w\in\T$.

\item
If $\Kvs=\KC$, then for $x = (\eta,\eta, w) \in \Kvs$,
the values $\ez(x)$ and $\eo(x)$ are independent of $w\in\T$.
\end{enumerate}
\end{lem}


\begin{proof}
Let $k = 0, 1$ and $\eta \in \sa$.

(1)
Suppose that $\Kvs\in\set{\KS,\Ks}$.
The function $s_k(\cdot, \eta, \cdot)$,
which sends $(z,w)\in q_1(\Kvs)\times\T$ to $s_k(z, \eta, w)$,
is continuous on $q_1(\Kvs) \times \T$.
Since $q_1(\Kvs)=\T$ if $\Kvs=\KS$, and $q_1(\Kvs)=\set{a}$ if $\Kvs=\Ks$,
the product space $q_1(\Kvs)\times\T$ is connected.
Then the image $s_k(q_1(\Kvs), \eta, \T)$ of $q_1(\Kvs) \times \T$
under the continuous mapping $s_k(\cdot, \eta, \cdot)$ is a connected subset
of $\set{\pm 1}$.
Then we deduce that the value $s_k(z, \eta, w)$ does not depend on
$z \in q_1(\Kvs)$ and $w \in \T$.

(2)
The function $s_k(\eta,\eta, \cdot)$, which sends
$w\in\T$ to $s_k(\eta,\eta,w)$, is continuous on $\T$.
The image $s_k(\eta,\eta,\T)$ of $\T$ under the mapping
$s_k(\eta,\eta,\cdot)$ is a connected subset of $\set{\pm 1}$.
Hence, the value $s_k(\eta,\eta,w)$ does not depend on $w \in \T$.
\end{proof}

By Lemmas~\ref{lem4.1}, \ref{lem4.2} and \ref{lem4.3}, we may write
$\phio(z, \eta, w) = \phio(z,\eta)$, $\psio(z, \eta, w) = \psio(z,\eta)$
and, for $k=0,1$, $s_k(z, \eta, w) = s_k(\eta)$ for $(z, \eta, w) \in \Kvs$. 
Then we can rewrite equality \eqref{lem4.1.0} as follows:
\begin{equation}\label{lem4.3.1}
\wh{T_0(f)}(z) + \wh{T_0(f)'}(\eta)w
= [\alo(x)\wh{f}(\phio(z,\eta))]^{\ez(\eta)}
+ [\alo(x) \wh{f'}(\psio(z,\eta))\varphio(x)]^{\ez(\eta)\eo(\eta)}
\end{equation}
for $f \in \Ao$ and $x = (z, \eta, w) \in \Kvs$.

We need to show that $\alo(z,\eta,w)$ is independent of $w\in\T$.
We use the following elementary result to prove it.


\begin{prop}\label{prop4.4}
Let $\la, \mu \in \C$. If $|\la + \mu w| = 1$ for all $w \in \T$, then $\la\mu=0$
and $|\la| + |\mu| = 1$.
\end{prop}


\begin{proof}
Suppose, on the contrary, that $\la\mu \neq 0$.
Choose $w_1 \in \T$ so that $\mu \, w_1 = \la|\mu||\la|^{-1}$, and set $w_2 = -w_1$. 
By hypothesis, $|\la + \mu \, w_1| = 1 = |\la + \mu \, w_2|$, that is, 
$$
\left| \la + \fr{\la|\mu|}{|\la|} \right| = 1 = \left| \la - \fr{\la|\mu|}{|\la|} \right|. 
$$
These equalities show that $|\la| + |\mu| = \bigl| |\la| - |\mu| \bigr|$. 
This implies that $\la = 0$ or $\mu = 0$, 
which contradicts the hypothesis that $\la\mu \neq 0$. Thus we obtain $\la\mu = 0$,
and then $|\la| + |\mu| = 1$.
\end{proof}


\subsection{Proof of Theorem~\ref{thm1}}

In this subsection, we give proof of Theorem~\ref{thm1}; to this end,
we consider the case when $\Kvs=\KS=\DS\times\T$
and $\Kvs=\KC=\DC\times\T$, where $\DS=\T\times\sa$ and
$\DC=\set{(\eta,\eta):\eta\in\sa}$.


\begin{lem}\label{lem4.5}
\begin{enumerate}
\item
The function $\ez$ as in Lemma~\ref{lem3.2} is constant
on $\sa$; we will write $\ez$ instead of $\ez(\eta)$.

\item
There exists a constant $c \in \T$ such that
\begin{enumerate}
\item
$\wh{T_0(\unit)}(z) = c^{\,\ez}$ for all $z \in \Db$,

\item
$\wh{T_0(i\unit)}(z) = i\ez \wh{T_0(\unit)}(z)$ for all $z \in \Db$,

\item
$\alo(x) = c$ for all $x \in\Kvs$.
\end{enumerate}
\end{enumerate}
\end{lem}


\begin{proof}
Let $\la \in \set{1,i}$. If we substitute $f = \la \unit \in \Ao$ into \eqref{lem4.3.1},
then
\begin{equation}\label{lem4.5.1}
\wh{T_0(\la\unit)}(z) + \wh{T_0(\la\unit)'}(\eta)w
= [\la\alo(x)]^{\ez(\eta)}
\end{equation}
for all $x = (z, \eta, w) \in\Kvs$.
We shall show that $\wh{T_0(\la\unit)'}=0$ on $\M$
for $\Kvs=\KS$ and $\Kvs=\KC$.

(i)
Let $\Kvs=\KS$.
Suppose, on the contrary, that there exists $\eta_0 \in \sa$ such that
$\wh{T_0(\la\unit)'}(\eta_0) \neq 0$. 
Let $z\in\T$.
By \eqref{lem4.5.1},
$|\wh{T_0(\la\unit)}(z) + \wh{T_0(\la\unit)'}(\eta_0)w| = 1$ for all $w \in \T$. 
Because $\wh{T_0(\la\unit)'}(\eta_0) \neq 0$, Proposition~\ref{prop4.4} shows that
$\wh{T_0(\la\unit)}(z) = 0$.
Since $z\in\T$ is arbitrary, we deduce that $\wh{T_0(\la\unit)} = 0$ on $\T$.
Since $T_0(\la\unit)$ is an analytic function on $\Di$,
it follows that $T_0(\la\unit)$ is identically zero on $\Di$.
This implies that $\wh{T_0(\la\unit)'}=0$ on $\M$,
which contradicts $\wh{T_0(\la\unit)'}(\eta_0)\neq 0$.
Hence, $\wh{T_0(\la\unit)'}(\eta)=0$ for all $\eta\in\sa$.
Since $\sa$ is a boundary for $\A$, we get $\wh{T_0(\la\unit)'}=0$ on $\M$.

(ii) Now, we suppose that $\Kvs=\KC$.
Choose $\eta_0\in\sa$ arbitrarily.
Equality \eqref{lem4.5.1} shows that
$|\wh{T_0(\la\unit)}(\eta_0) + \wh{T_0(\la\unit)'}(\eta_0)w| = 1$
for all $w \in \T$. By Proposition~\ref{prop4.4},
$$
\wh{T_0(\la\unit)}(\eta_0) \wh{T_0(\la\unit)'}(\eta_0) = 0
\qbox{and}\q
|\wh{T_0(\la\unit)}(\eta_0)| + |\wh{T_0(\la\unit)'}(\eta_0)| = 1.
$$
Because $\eta_0\in\sa$ is arbitrarily chosen, we deduce that
$\wh{T_0(\la\unit)}\,\wh{T_0(\la\unit)'} = 0$
and $|\wh{T_0(\la\unit)}| + |\wh{T_0(\la\unit)'}| = 1$
on $\sa$.
Since $\sa$ is a boundary for $\A$,
we see that $\wh{T_0(\la\unit)}\,\wh{T_0(\la\unit)'} = 0$ on $\M$.
Hence $T_0(\la\unit)T_0(\la\unit)' = 0$ on $\Di$.
Both $T_0(\la\unit)$ and $T_0(\la\unit)'$ are analytic on the connected set $\Di$,
and thus $T_0(\la\unit) = 0$ on $\Di$, or $T_0(\la\unit)' = 0$ on $\Di$.
If $T_0(\la\unit) = 0$ on $\Di$, then we would have $T_0(\la\unit)' = 0$ on $\Di$,
and hence $\wh{T_0(\la\unit)} = \wh{T_0(\la\unit)'} = 0$ on $\M$.
This contradicts $|\wh{T_0(\la\unit)}| + |\wh{T_0(\la\unit)'}| = 1$ on $\sa$,
and consequently we obtain $T_0(\la\unit)' = 0$ on $\Di$.
This implies that $\wh{T_0(\la\unit)'}$ is identically zero on $\M$,
as is claimed.

Since $\Di$ is connected, we see that $T_0(\la\unit)$ is constant on $\Di$.
We set $c_\la = T_0(\la\unit)$, and then $c_\la = [\la\alo(x)]^{\ez(\eta)}$
for all $x = (z,\eta, w) \in \Kvs$ by \eqref{lem4.5.1}. Then we obtain $|c_\la| = 1$.
It follows that $c_i = [i\alo(x)]^{\ez(\eta)} = [i]^{\ez(\eta)} [\alo(x)]^{\ez(\eta)} = i\ez(\eta)c_1$,
and hence $c_i = i\ez(\eta)c_1$ for all $\eta \in \sa$.
This shows that $\ez$ is a constant function on $\sa$.
We set $c = c_1^{\,\ez}$, and then $\wh{T_0(\unit)} = c_1 = c^{\,\ez}$ and
$\wh{T_0(i\unit)} = c_i = i\ez c^{\,\ez}$ on $\sa$.
Furthermore, $\alo = c_1^{\,\ez} = c$ on $\Kvs$.
Since $\sa$ is a boundary for $\A$, we have $\wh{T_0(\unit)}(z)=c^{\,\ez}$ and $\wh{T_0(i\unit)}(z)=i\ez\wh{T_0(\unit)}(z)$ for $z\in\Db$.
\end{proof}

By Lemma~\ref{lem4.5}, equality \eqref{lem4.3.1} is reduced to
\begin{equation}\label{lem4.6.0}
\wh{T_0(f)}(z) + \wh{T_0(f)'}(\eta)w \\
= [c\wh{f}(\phio(z,\eta))]^{\ez} + [c\wh{f'}(\psio(z,\eta))\varphio(x)]^{\ez\eo(\eta)}
\end{equation}
for every $f \in \Ao$ and $x = (z, \eta, w)\in\Kvs$.
Then we show that $\phio(z,\eta)$ is constant for $\eta\in\sa$.


\begin{lem}\label{lem4.6}
Let $c \in \T$ be the constant as in Lemma~\ref{lem4.5}. Then 
$\wh{T_0(\id)}(z) = [c\phio(z,\eta)]^{\ez}$
for all $(z,\eta)\in\Dvs$
with $\Dvs=\DS$ or $\Dvs=\DC$.
\end{lem}


\begin{proof}
Let $(z_0,\eta_0)\in\Dvs$. We set $\xi_0 = \phio(z_0,\eta_0)$ and 
$g = \id - \xi_0 \unit \in \Ao$.
Then $\wh{g}(\xi_0) = 0$ and $g' = \unit$, and thus, by \eqref{lem4.6.0},
\begin{equation}\label{lem4.6.1}
\wh{T_0(g)}(z_0) + \wh{T_0(g)'}(\eta_0)w
= [c\varphio(z_0, \eta_0, w)]^{\ez\eo(\eta_0)}
\end{equation}
for all $w \in \T$.
Hence $|\wh{T_0(g)}(z_0) + \wh{T_0(g)'}(\eta_0)w| = 1$ for all $w \in \T$. 
By Proposition~\ref{prop4.4}, we obtain
$\wh{T_0(g)}(z_0)\wh{T_0(g)'}(\eta_0) = 0$.
We shall prove that $\wh{T_0(g)}(z_0) = 0$.
Suppose, on the contrary, that $\wh{T_0(g)}(z_0) \neq 0$, and then
$\wh{T_0(g)'}(\eta_0) = 0$. By \eqref{lem4.6.1}, we have
$$
\wh{T_0(g)}(z_0) = [c\varphio(z_0, \eta_0, w)]^{\ez\eo(\eta_0)}
$$
for all $w \in \T$.
By the surjectivity of $T_0$, there exists $h \in \Ao$ such that 
$\wh{T_0(h)}(z_0) = 0$ and $\wh{T_0(h)'}(\eta_0) = 1$. 
Applying these two equalities to \eqref{lem4.6.0}, we get
\begin{align*}
w
&=
\wh{T_0(h)}(z_0) + \wh{T_0(h)'}(\eta_0)w
=
[c\wh{h}(\phio(z_0,\eta_0))]^{\ez}
+ [c\varphio(z_0, \eta_0, w)]^{\ez\eo(\eta_0)} [\wh{h'}(\psio(z_0,\eta_0))]^{\ez\eo(\eta_0)} \\
&=
[c\wh{h}(\phio(z_0,\eta_0))]^{\ez}
+ \wh{T_0(g)}(z_0) [\wh{h'}(\psio(z_0,\eta_0))]^{\ez\eo(\eta_0)}
\end{align*}
for all $w \in \T$. The rightmost hand side of the above equalities is independent of $w$,
and thus we arrive at a contradiction.
Consequently, $\wh{T_0(g)}(z_0) = 0$ as is claimed.

Since $T_0$ is real linear, 
$0 = \wh{T_0(g)}(z_0) = \wh{T_0(\id)}(z_0) - \wh{T_0(\xi_0 \unit)}(z_0)$, 
and hence $\wh{T_0(\id)}(z_0) = \wh{T_0(\xi_0 \unit)}(z_0)$.
Here, we note that $\wh{T_0(\unit)} = c^{\,\ez}=[c]^{\ez}$
and $\wh{T_0(i\unit)} = i\ez \, \wh{T_0(\unit)}$ by Lemma~\ref{lem4.5}.
These imply that
\begin{align*}
\wh{T_0(\id)}(z_0) 
&=
\wh{T_0(\xi_0\unit)}(z_0)
=
(\Re \xi_0) \, \wh{T_0(\unit)}(z_0) + (\Im \xi_0) \, \wh{T_0(i\unit)}(z_0) \\
&=
(\Re \xi_0 + i\ez \, \Im \xi_0) \wh{T_0(\unit)}(z_0)
=[\xi_0]^{\ez} [c]^{\ez} = [c\phio(z_0,\eta_0)]^{\ez},
\end{align*}
where we have used the real linearity of $T_0$.
Since $(z_0,\eta_0)\in\Dvs$ is arbitrarily chosen, we get
$\wh{T_0(\id)}(z) = [c\phio(z,\eta)]^{\ez}$ for all $(z,\eta)\in\Dvs$.
\end{proof}

Finally, we investigate the function $\varphio$ on $\Kvs$.
In fact, we prove that $\varphio(z,\eta,w)$ is constant for $z$.
More explicitly, we have the following result on $\varphio$.


\begin{lem}\label{lem4.7}
For each $(z, \eta, w) \in\Kvs$,
$$
\wh{T_0(\id)'}(\eta) = [c\varphio(z, \eta, 1)]^{\ez\eo(\eta)}
\qbox{and}\q
\varphio(z,\eta, w) = w^{\ez\eo(\eta)}\varphio(z, \eta, 1).
$$
\end{lem}


\begin{proof}
Let $(z,\eta,w) \in\Kvs$.
Then $\wh{T_0(\id)}(z) = [c\phio(z,\eta)]^{\ez}$ by Lemma~\ref{lem4.6}.
Substituting this equality and $f = \id$ into \eqref{lem4.6.0}, we obtain
$\wh{T_0(\id)'}(\eta)w = [c\varphio(z,\eta,w)]^{\ez\eo(\eta)}$.
If we take $w = 1$ in the last equality, then
$\wh{T_0(\id)'}(\eta) = [c\varphio(z, \eta, 1)]^{\ez\eo(\eta)}$.
Note that $\wh{T_0(\id)'}(\eta)\neq 0$, since $c, \varphio(z,\eta,1) \in \T$
(see Definition~\ref{defn2} and Lemma~\ref{lem4.5}).
We have
$$
w = \fr{\wh{T_0(\id)'}(\eta)w}{\wh{T_0(\id)'}(\eta)}
= \fr{[c\varphio(z, \eta, w)]^{\ez\eo(\eta)}}{[c\varphio(z, \eta, 1)]^{\ez\eo(\eta)}}
= \fr{[\varphio(z, \eta, w)]^{\ez\eo(\eta)}}{[\varphio(z, \eta, 1)]^{\ez\eo(\eta)}},
$$
and hence $[\varphio(z, \eta, w)]^{\ez\eo(\eta)} = w[\varphio(z, \eta, 1)]^{\ez\eo(\eta)}$.
We conclude $\varphio(z, \eta, w) = w^{\ez\eo(\eta)}\varphio(z, \eta, 1)$
for all $(z,\eta,w)\in\Kvs$.
\end{proof}

Now we are ready to prove Theorem~\ref{thm1} for the characterizations
of the isometry $T_0$ on $\Ao$ with the norms $\VS{\cdot}$ and $\VC{\cdot}$.
Then $\Kvs=\KS=\DS\times\T$ and $\Kvs=\KC=\DC\times\T$
with $\DS=\T\times\sa$ and $\DC=\set{(\eta,\eta):\eta\in\sa}$.


\begin{proof}[\bf\textit{Proof of Theorem~\ref{thm1}}]
Fix arbitrary $f \in \Ao$ and $(z_0,\eta_0)\in\Dvs$.
We have $\varphio(z_0, \eta_0, w) = w^{\ez\eo(\eta_0)}\varphio(z_0, \eta_0, 1)$
for all $w \in \T$ by Lemma~\ref{lem4.7}.
Applying Lemmas~\ref{lem4.6} and \ref{lem4.7}, we may write
$\phio(z,\eta) = \phio(z)$ and 
$\varphio(z, \eta, w) =w^{\ez\eo(\eta)}\varphio(\eta)$ for all $(z, \eta, w) \in\Kvs$.
Now equality \eqref{lem4.6.0} is reduced to
\[
\wh{T_0(f)}(z_0) + \wh{T_0(f)'}(\eta_0)w
= [c\wh{f}(\phio(z_0))]^{\ez} + w[c\varphio(\eta_0)]^{\ez\eo(\eta_0)}
[\wh{f'}(\psio(z_0,\eta_0))]^{\ez\eo(\eta_0)}
\]
for all $w \in \T$. The above equality holds for every $w \in \T$, and then
$\wh{T_0(f)}(z_0)=[c\wh{f}(\phio(z_0))]^{\ez}$ and
$\wh{T_0(f)'}(\eta_0)
=[c\varphio(\eta_0)]^{\ez\eo(\eta_0)} [\wh{f'}(\psio(z_0,\eta_0))]^{\ez\eo(\eta_0)}$.
Since $(z_0,\eta_0)\in\Dvs$ is arbitrary, we get
\begin{align}
\wh{T_0(f)}(z)
&=
[c\wh{f}(\phio(z))]^{\ez}, \label{thm1.2}\\
\wh{T_0(f)'}(\eta)
&=
[c\varphio(\eta)]^{\ez\eo(\eta)} [\wh{f'}(\psio(z,\eta))]^{\ez\eo(\eta)}
\label{thm1.3}
\end{align}
for all $(z,\eta)\in\Dvs$.
Set $\rho = c^{-\ez}\,T_0(\id) \in A(\Db)$, and then 
$c^{\,\ez}\wh{\rho}(z) = \wh{T_0(\id)}(z) = [c\phio(z)]^{\ez}$ for all
$z\in q_1(\Kvs)$ by \eqref{thm1.2}.
This shows that
\begin{equation}\label{thm1.4}
\phio(z) = [\wh{\rho}(z)]^{\ez}
\end{equation}
for all $z\in q_1(\Kvs)$.
Consequently, by \eqref{thm1.2} and \eqref{thm1.4},
\begin{equation}\label{thm1.5}
\wh{T_0(f)}(z) = [c\wh{f}([\wh{\rho}(z)]^{\ez})]^{\ez}
\end{equation}
for all $z\in q_1(\Kvs)$.

If $\Kvs=\KS$, then $q_1(\Kvs)=\T$.
Since $\rho$ is an analytic function on $\Di$, the maximum modulus principle
shows that equality \eqref{thm1.5} holds for all $z \in \Db$.
Note that $\rho \in A(\Db)$ satisfies $|\wh{\rho}(z)| = |\phio(z)| = 1$
for all $z\in q_1(\Kvs)=\T$
by \eqref{thm1.4},
and thus $\wh{\rho}(\Di)\subset\Db$ by the maximum modulus principle.

If $\Kvs=\KC$, then $q_1(\Kvs)=\sa$. Since $\sa$ is a boundary for $\A$,
we see that \eqref{thm1.5} holds for all $\eta\in\M$. In particular,
\eqref{thm1.5} is valid for all $z\in\Di$.
By the continuity of $\wh{T_0(f)}$, $\wh{f}$ and $\wh{\rho}$ on $\Db$,
we observe that equality \eqref{thm1.5} is true for all $z\in\Db$.
Note that $\rho \in A(\Db)$ satisfies $|\wh{\rho}(\eta)| = |\phio(\eta)| = 1$
for all $\eta \in \sa$, and thus $\wh{\rho}(\M) \subset \Db$.
Thus, we have $\wh{\rho}(\Di)\subset\Db$.

We now prove that $\wh{\rho}$ is injective on $\Di$.
Suppose that $\wh{\rho}(z_1) = \wh{\rho}(z_2)$ for $z_1, z_2 \in \Di$.
Since $T_0$ is surjective, there exists $f_0 \in \Ao$ such that
$T_0(f_0) = \id \in \Ao$.
Since \eqref{thm1.5} holds for $z\in\Di$, we obtain
\[
z_1 = \wh{T_0(f_0)}(z_1) = [c\wh{f_0}([\wh{\rho}(z_1)]^{\ez})]^{\ez}
= [c\wh{f_0}([\wh{\rho}(z_2)]^{\ez})]^{\ez} = \wh{T_0(f_0)}(z_2) = z_2,
\]
and thus $\wh{\rho}$ is injective on $\Di$, as is claimed.
Then $\wh{\rho}$ is a non-constant function on $\Di$.
By the maximum modulus principle, we see that $\wh{\rho}(\Di)\subset\Di$.

Since $T_0^{-1}$ is a surjective real linear isometry on $\Ao$ as well,
the above arguments can be applied to $T_0^{-1}$.
Then there exist $d \in \T$, $\theta \in A(\Db)$ with $\wh{\theta}(\Di) \subset \Di$
and $\et \in \set{\pm 1}$ such that
\begin{equation}\label{thm1.6}
\wh{T_0^{-1}(f)}(z) = [d\wh{f}([\wh{\theta}(z)]^{\et})]^{\et}
\end{equation}
for all $f \in \Ao$
and $z \in \Di$. Let $z \in \Di$. If we substitute $f = T_0(\unit)$ into \eqref{thm1.6},
then we have $1 = [d\wh{T_0(\unit)}([\wh{\theta}(z)]^{\et})]^{\et}$.
As $\wh{T_0(\unit)} = c^{\,\ez}$ by \eqref{thm1.5}, we obtain $1 = [dc^{\,\ez}]^{\et}$.
Substituting $f = T_0(\id)$ into \eqref{thm1.6} to get
$z = [d\wh{T_0(\id)}([\wh{\theta}(z)]^{\et})]^{\et}$.
By \eqref{thm1.5}, $\wh{T_0(\id)}(z) = c^{\,\ez}\wh{\rho}(z)$, and hence
$$
z = [d\wh{T_0(\id)}([\wh{\theta}(z)]^{\et})]^{\et} = [dc^{\,\ez} \wh{\rho}([\wh{\theta}(z)]^{\et})]^{\et}
= [\wh{\rho}([\wh{\theta}(z)]^{\et})]^{\et}, 
$$
where we have used that $1 = [dc^{\,\ez}]^{\et}$.
Because $z\in\Di$ is arbitrarily chosen, we get
$z = [\wh{\rho}([\wh{\theta}(z)]^{\et})]^{\et}$ for all $z \in \Di$.
This shows that $\Di \subset \wh{\rho}(\Di)$, and therefore,
$\wh{\rho}(\Di) = \Di$.

From the above arguments, we see that
$\rho = c^{-\ez}\,T_0(\id) \in A(\Db)$ is an analytic function on $\Di$,
which is a homeomorphism on $\Di$
as well. It is well-known (cf. \cite[Theorem~12.6]{rud}) that for such a function $\rho$
there exist $\la \in \T$ and $b \in \Di$ such that
\begin{equation}\label{thm1.7}
\wh{\rho}(z) = \la \, \fr{z-b}{1-\ov{b}z}
\qq (z \in \Di).
\end{equation}

Because $\wh{T_0(\id)}(z) = c^{\,\ez}\wh{\rho}(z)$ for $z \in \Di$,
we have $T_0(\id)'(z) = c^{\,\ez}\rho'(z)$ for $z \in \Di$.
Differentiating both sides of \eqref{thm1.7},
we see that $\rho'(z) = \la(1-|b|^2)/(1-\ov{b}z)^2$ for $z \in \Di$.
Hence $c^{-\ez}T_0(\id)' = \rho'$ belongs to $A(\Db)$.
Therefore, $\wh{\rho'}(\xi)=\wh{\rho'}(\eta)$ for all
$\xi\in\T$ and $\eta\in\M_\xi$. Substituting $f = \id$ into \eqref{thm1.3},
we get $|\wh{T_0(\id)'}| = |c\varphio| = 1$ on $\sa$. Hence, $|\wh{\rho'}|=1$ on $\sa$.
We see that $\sa\cap\M_\xi\neq\emptyset$ for all $\xi\in\T$. In fact, let
$\xi_0\in\T$, and set $u(z)=(\ov{\xi_0}\,z+1)/2\in A(\Db)$ for $z\in\Di$.
Then $\wh{u}(\eta)=\wh{u}(\xi)$ for all $\xi\in\T$ and $\eta\in\M_\xi$.
Note that $\wh{u}(\xi_0)=1$ and
$|\wh{u}(\xi)|<1$ for all $\xi\in\T\setminus\set{\xi_0}$.
Since $\sa$ is a boundary for $\A$, 
we obtain $\sa\cap\M_{\xi_0}\neq\emptyset$.
This proves $\sa\cap\M_\xi\neq\emptyset$ for all $\xi\in\T$.
For each $z\in\T$, we can find $\eta\in\sa\cap\M_z$.
Then, $|\wh{\rho'}(z)|=|\wh{\rho'}(\eta)|=1$, since $\rho'\in A(\Db)$.
Consequently,
$|\wh{\rho'}(z)| = 1$ for all $z \in \T$, that is, $1-|b|^2 = |1-\ov{b}z|^2$
for all $z \in \T$.
This implies that $b = 0$: In fact, if $b \neq 0$, then we substitute $z = \pm b/|b|$
into the equality $1-|b|^2 = |1-\ov{b}z|^2$ to obtain
$$
(1-|b|)^2 = \left| 1 - \ov{b} \, \fr{b}{|b|} \right|^2 = 1-|b|^2
= \left| 1 - \ov{b} \, \fr{-b}{|b|} \right|^2 = (1+|b|)^2,
$$
which is impossible. Therefore, $b = 0$ and consequently $\wh{\rho}(z) = \la z$
for all $z \in \Di$.
By \eqref{thm1.5} we conclude that $T_0(f)(z) = cf(\la z)$ for all $f \in \Ao$
and $z \in \Di$,
or $T_0(f)(z) = \ov{cf(\ov{\la z})}$ for all $f \in \Ao$ and $z \in \Di$.
By \eqref{T_0}, $T_0 = T - T(0)$, and thus 
$T(f)(z)=T(0)(z)+cf(\la z)$ for all $f\in\Ao$ and $z\in\Di$, or
$T(f)(z)=T(0)(z)+\ov{cf(\ov{\la z})}$ for all $f\in\Ao$ and $z\in\Di$.

Conversely, if $T$ is one of the above forms, then it is easy to check that $T$ is a
surjective isometry on $\Ao$.
\end{proof}


\subsection{Proof of Theorem~\ref{thm2}}

In this subsection, we consider the case when $\Kvs=\Ks=\set{a}\times\sa\times\T$
for $a \in \Di$.
Recall that
\[
\wh{T_0(f)}(z) + \wh{T_0(f)'}(\eta)w
= [\alo(x)\wh{f}(\phio(z,\eta))]^{\ez(\eta)}
+ [\alo(x) \wh{f'}(\psio(z,\eta))\varphio(x)]^{\ez(\eta)\eo(\eta)}
\]
for all $x = (z, \eta, w) \in\Ks$ by \eqref{lem4.6.0}.
By definition, $\phio$ is a map from
$\Ks$ to $q_1(\Ks)=\set{a}$.
Hence $\phio$ is the constant function which takes the value only $a$.
In addition, since $q_1(\Ks)=\set{a}$, we may write 
\begin{equation}\label{lem4.9.-1}
\psio(z,\eta) = \psio(\eta)
\q\mbox{and}\q
\varphio(z,\eta,w) = \varphio(\eta,w)
\end{equation}
for $(z, \eta) \in \set{a} \times \sa$ and $w \in \T$.
Now we obtain the following equality
\begin{equation}\label{lem4.9.0}
T_0(f)(a) + \wh{T_0(f)'}(\eta)w
= [\alo(x)f(a)]^{\ez(\eta)} + [\alo(x) \wh{f'}(\psio(\eta))\varphio(\eta,w)]^{\ez(\eta)\eo(\eta)}
\end{equation}
for all $f \in \Ao$ and $x = (a, \eta, w) \in \set{a} \times \sa \times \T$.

We need to prove that
$\wh{T_0(f)'}(\eta)w=[\alo(x) \wh{f'}(\psio(\eta))\varphio(\eta,w)]^{\ez(\eta)\eo(\eta)}$.
We first show that $\alo(a,\eta,w)$ is independent of $w\in\T$.


\begin{lem}\label{lem4.9}
There exists $c_0 \in \T$ such that $\wh{T_0(\unit)} = c_0$ on $\Db$
and $\alo(a, \eta, w) = [c_0]^{\ez(\eta)} = \alo(a, \eta, 1)$
for all $(\eta, w) \in \sa \times \T$.
\end{lem}


\begin{proof}
Let $g = \id - a \unit$, and then $g'=\unit$.
If we apply $f = \unit, g \in \Ao$ to \eqref{lem4.9.0}, then
\begin{align}
T_0(\unit)(a) + \wh{T_0(\unit)'}(\eta)w
&= 
[\alo(a, \eta, w)]^{\ez(\eta)},
\label{lem4.9.1} \\
T_0(g)(a) + \wh{T_0(g)'}(\eta)w
&=
[\alo(a, \eta, w)\varphio(\eta, w)]^{\ez(\eta)\eo(\eta)}
\label{lem4.9.2}
\end{align}
for all $(\eta,w) \in \sa \times \T$.
We shall prove that $T_0(\unit)(a) \neq 0$. To this end, we assume that
$T_0(\unit)(a) = 0$,
and then $\wh{T_0(\unit)'}(\eta)w = [\alo(a, \eta, w)]^{\ez(\eta)}$ for all 
$(\eta,w) \in \sa \times \T$.
Substituting this equality and \eqref{lem4.9.2} into \eqref{lem4.9.0} to obtain
\[
T_0(f)(a) + \wh{T_0(f)'}(\eta)w
= \wh{T_0(\unit)'}(\eta)w [f(a)]^{\ez(\eta)} + \{ T_0(g)(a) + \wh{T_0(g)'}(\eta)w \} 
[\wh{f'}(\psio(\eta))]^{\ez(\eta)\eo(\eta)}
\]
for all $f \in \Ao$ and $(\eta,w) \in \sa \times \T$.
Since $w\in\T$ is arbitrary, we get
\begin{equation}\label{lem4.9.3}
T_0(f)(a) = T_0(g)(a) [\wh{f'}(\psio(\eta))]^{\ez(\eta)\eo(\eta)}
\end{equation}
for all $f \in \Ao$ and $\eta \in \sa$.
Taking $f = \id^2 \in \Ao$ in \eqref{lem4.9.3}, we have
\begin{equation}\label{lem4.9.4}
T_0(\id^2)(a) = T_0(g)(a) [2\,\wh{\id}(\psio(\eta))]^{\ez(\eta)\eo(\eta)}
\end{equation}
for all $\eta \in \sa$.
Let $j\in\set{\pm 1}$.
Since $\sa\cap\M_j\neq\emptyset$, we can choose
$\eta_j\in\sa\cap\M_j$. Then $\wh{\id}(\eta_j)=j$.
By Lemma~\ref{lem3.9} (see, also Lemma~\ref{lem4.2}),
$\psio \colon \sa \to \sa$ is surjective, and thus,
there exists $\zeta_j\in\sa$ such that $\psio(\zeta_j)=\eta_j$.
We thus obtain $\wh{\id}(\psio(\zeta_j))=j$.
Substituting $\eta=\zeta_j$ into \eqref{lem4.9.4} to get
$T_0(\id^2)(a)=2jT_0(g)(a)$, which yields $T_0(g)(a) = 0$.
We obtain $T_0(f)(a) = 0$ for all $f \in \Ao$ by \eqref{lem4.9.3}.
This is impossible since $T_0$ is surjective. Consequently, $T_0(\unit)(a) \neq 0$,
as is claimed.

We see that $\wh{T_0(\unit)'}(\eta) = 0$ for all $\eta \in \sa$
by equality \eqref{lem4.9.1} with Proposition~\ref{prop4.4}.
Since $\sa$ is a boundary for $\A$, we have $\wh{T_0(\unit)'} = 0$ on $\M$.
Then $T_0(\unit)$ is constant on $\Di$, say $c_0 \in \C$.
Since $\Di$ is dense in $\M$, we obtain $\wh{T_0(\unit)} = c_0$ on $\M$.
In particular, $\wh{T_0(\unit)}=c_0$ on $\Db$.
Substituting $\wh{T_0(\unit)'}(\eta) = 0$ into \eqref{lem4.9.1} to get
$c_0 = T_0(\unit)(a) = [\alo(a, \eta, w)]^{\ez(\eta)}$ for all $(\eta, w) \in \sa \times \T$,
and thus $c_0 \in \T$.
Hence $\alo(a,\eta,w) = [c_0]^{\ez(\eta)} = \alo(a,\eta,1)$.
\end{proof}

By Lemma~\ref{lem4.9},
$\alo(a,\eta,w) = [c_0]^{\ez(\eta)}$ for all $(\eta,w) \in \sa \times \T$.
Equality \eqref{lem4.9.0} is reduced to
\begin{equation}\label{lem4.10.0}
T_0(f)(a) + \wh{T_0(f)'}(\eta)w
= c_0[f(a)]^{\ez(\eta)}
+ [c_0]^{\eo(\eta)} [\wh{f'}(\psio(\eta))\varphio(\eta, w)]^{\ez(\eta)\eo(\eta)}
\end{equation}
for every $f \in \Ao$ and $(\eta, w) \in \sa \times \T$.

In the next lemma, we determine $\varphio(\eta,w)$ as a function of variable 
$w\in\T$.
Then we derive the form of $\wh{T_0(f)'}(\eta)$ from \eqref{lem4.10.0}.


\begin{lem}\label{lem4.10}
\begin{enumerate}
\item
$\wh{T_0(\id_a)'}(\eta) = [c_0]^{\eo(\eta)} [\varphio(\eta, 1)]^{\ez(\eta)\eo(\eta)}$
for all $\eta \in \sa$, where $\id_a=\id-a\unit$.

\item
$\varphio(\eta, w) = \varphio(\eta, 1)w^{\ez(\eta)\eo(\eta)}$
for all $(\eta, w) \in \sa \times \T$.
\end{enumerate}
\end{lem}


\begin{proof}
Let $\eta_0 \in \sa$.
Taking $f = \id _a$ in \eqref{lem4.10.0}, we obtain
\begin{equation}\label{lem4.10.1}
T_0(\id_a)(a) + \wh{T_0(\id_a)'}(\eta_0)w = 
[c_0]^{\eo(\eta_0)} [\varphio(\eta_0, w)]^{\ez(\eta_0)\eo(\eta_0)}
\end{equation}
for every $w \in \T$. 
Then $|T_0(\id_a)(a) + \wh{T_0(\id_a)'}(\eta_0)w| = 1$ for all $w \in \T$.
We shall prove that $T_0(\id_a)(a) = 0$. Suppose, on the contrary,
that $T_0(\id_a)(a) \neq 0$.
Proposition~\ref{prop4.4} shows that $\wh{T_0(\id_a)'}(\eta_0) = 0$.
Then, from \eqref{lem4.10.1} we have
$T_0(\id_a)(a) = [c_0]^{\eo(\eta_0)} [\varphio(\eta_0, w)]^{\ez(\eta_0)\eo(\eta_0)}$
for all $w \in \T$.
Since $T_0$ is surjective, there exists $g \in \Ao$ such that $T_0(g)(a) = 0$
and $\wh{T_0(g)'}(\eta_0) = 1$.
Applying these equalities and 
$T_0(\id_a)(a) = [c_0]^{\eo(\eta_0)}[\varphio(\eta_0,w)]^{\ez(\eta_0)\eo(\eta_0)}$
to \eqref{lem4.10.0}, we get
\[
w 
=T_0(g)(a) + \wh{T_0(g)'}(\eta_0)w
=c_0[g(a)]^{\ez(\eta_0)} + T_0(\id_a)(a)\, [\wh{g'}(\psio(\eta_0))]^{\ez(\eta_0)\eo(\eta_0)}
\]
for all $w \in \T$. 
This is impossible since the rightmost hand side of the above equalities is independent
of $w \in \T$. Consequently, $T_0(\id_a)(a) = 0$.
Because $\eta_0\in\sa$ is arbitrarily chosen, it follows from \eqref{lem4.10.1} that
$\wh{T_0(\id_a)'}(\eta)w = [c_0]^{\eo(\eta)} [\varphio(\eta, w)]^{\ez(\eta)\eo(\eta)}$
for all $(\eta, w) \in \sa \times \T$.
Taking $w = 1$ in the last equality, we get
$\wh{T_0(\id_a)'}(\eta) =  [c_0]^{\eo(\eta)} 
[\varphio(\eta, 1)]^{\ez(\eta)\eo(\eta)}$ for all $\eta \in \sa$.

Note that $\wh{T_0(\id_a)'}(\eta)\neq 0$ for $\eta \in \T$. We thus obtain
$$
w = \fr{\wh{T_0(\id_a)'}(\eta)w}{\wh{T_0(\id_a)'}(\eta)}
= \fr{[\varphio(\eta, w)]^{\ez(\eta)\eo(\eta)}}{[\varphio(\eta, 1)]^{\ez(\eta)\eo(\eta)}},
$$
and therefore, $\varphio(\eta, w) = w^{\ez(\eta)\eo(\eta)}\varphio(\eta, 1)$
for all $(\eta, w) \in \sa \times \T$.
\end{proof}

We are in a position to prove Theorem~\ref{thm2}.
We need to investigate $\psio$ to characterize the isometry $T_0$ on $\Ao$
for $\Ks=\set{a}\times\sa\times\T$.
The difficulty is to extend the map $\psio\colon\sa\to\sa$
to a homeomorphism defined on $\M$.
The main idea of its proof is to induce a surjective real linear isometry
on the uniform algebra $\A$.
We then induce a real algebra isometric isomorphism on $\A$.
It is known that such an isometric isomorphism is represented as
a combination of a composition operator induced by a homeomorphism on $\M$
and the complex conjugate.
Then we prove that $\psio$ is extended to $\M$.


\begin{proof}[\bf\textit{Proof of Theorem~\ref{thm2}}]
For simplicity of notation, we write $\varphio(\eta, 1) = \varphio(\eta)$.
By Lemma~\ref{lem4.10}, equality \eqref{lem4.10.0} is reduced to
$$
T_0(f)(a) + \wh{T_0(f)'}(\eta)w
= c_0[f(a)]^{\ez(\eta)} + w [c_0]^{\eo(\eta)} 
[\varphio(\eta)\wh{f'}(\psio(\eta))]^{\ez(\eta)\eo(\eta)}
$$
for all $f \in \Ao$ and $(\eta,w) \in \sa \times \T$.
Since $w\in\T$ is arbitrary, we obtain
\begin{equation}\label{thm2.1}
\wh{T_0(f)'}(\eta) = [c_0]^{\eo(\eta)} [\varphio(\eta) \wh{f'}(\psio(\eta))]^{\ez(\eta)\eo(\eta)}
\end{equation}
and $T_0(f)(a) = c_0[f(a)]^{\ez(\eta)}$ for all $f \in \Ao$ and $\eta \in \sa$.
Substituting $f = i \unit \in \Ao$ into the last equality, we get
$T_0(i \unit)(a) = c_0[i]^{\ez(\eta)}=i\ez(\eta)c_0$ for all $\eta \in \sa$.
This shows that $s_0$ is a constant function;
we may and do write $\ez(\eta) = \ez$. Then we have
\begin{equation}\label{thm2.2}
T_0(f)(a) = c_0[f(a)]^{\ez}
\end{equation}
for all $f \in \Ao$. For each $v \in \H$, we define $I_a(v)$ by 
\[
I_a(v)(z) = \int_{[a,z]} v(\zeta)\,d\zeta
\qq (z \in \Di),
\]
where $[a,z]$ denotes the straight line interval from $a$ to $z$.
Then $I_a(v)$ belongs to $\H$ with $I_a(v)(a) = 0$ and
\begin{equation}\label{thm2.3}
I_a(v)' = v \qq\mbox{on}\ \Di.
\end{equation}
Hence $I_a(v) \in \Ao$ and $\wh{I_a(v)'} = \wh{v}$ on $\M$.
We may regard $I_a$ as a complex linear operator from $\H$ to $\Ao$.

Recall that $\A = \set{\wh{v} \in C(\M) : v \in \H}$. Define
$\W \colon \A \to \A$ by
\begin{equation}\label{thm2.4}
\W(\wh{v})(\eta) = \wh{\{T_0(I_a(v))\}'}(\eta)
\qq (v \in \H, \ \eta \in \M).
\end{equation}
Since the Gelfand transform from $\H$ onto $\A$ is an isometric isomorphism,
we see that the mapping $\W$ is well-defined.
$$ 
\begin{CD}
\A @>{\W}>> \A \\
@A{\wh{\cdot}}AA
@AA{\wh{\cdot}}A \\ 
\H @. \H \\
@V{I_a}VV 
@VV{I_a}V \\ 
\Ao @>>{T_0}> \Ao
\end{CD} 
$$
Because $T_0 \colon \Ao \to \Ao$ is real linear, so is $\W \colon \A \to \A$.
Recall that $\wh{I_a(v)'} = \wh{v}$ on $\M$ for $v \in \H$.
Substituting $f = I_a(v)$ into \eqref{thm2.1} to obtain
$\wh{\{T_0(I_a(v))\}'}(\eta) = [c_0]^{\eo(\eta)} [\varphio(\eta)\wh{v}(\psio(\eta))]^{\ez\eo(\eta)}$
for all $v \in \H$ and $\eta \in \sa$, where we have used that $\ez$ is a constant function.
The last equality, combined with \eqref{thm2.4}, shows that
$$
\W(\wh{v})(\eta) 
= [c_0]^{\eo(\eta)} [\varphio(\eta) \wh{v}(\psio(\eta))]^{\ez\eo(\eta)}
$$
for $\wh{v} \in \A$ and $\eta \in \sa$.
Note that $c_0, \varphio(\eta) \in \T$ by definition
(see Lemma~\ref{lem4.9} and Definition~\ref{defn2}),
and thus $|\W(\wh{v})(\eta)| = |\wh{v}(\psio(\eta))|$ for all $\wh{v} \in \A$ and $\eta \in \sa$.
Since $\psio \colon \sa \to \sa$ is surjective by Lemma~\ref{lem3.9} and \eqref{lem4.9.-1},
we see that
$\V{\W(\wh{v})}_\M = \V{\wh{v}}_\M$ for all $\wh{v} \in \A$.
Thus $\W$ is a real linear isometry on $(\A, \DV{\cdot}_\M)$.

We shall prove that
\begin{equation}\label{thm2.5}
\wh{\{T_0(I_a(f'))\}'} = \wh{T_0(f)'}
\end{equation}
on $\M$ for each $f\in\Ao$.
Let $f \in \Ao$, and then $\{I_a(f')\}' = f'$ on $\Di$ by \eqref{thm2.3}.
Thus $I_a(f') - f$ is constant on $\Di$. 
By the definition of $I_a$, we have $I_a(f')(a) = 0$, and hence
\begin{equation}\label{thm2.12}
I_a(f') = f - f(a)\unit
\qq\mbox{on $\Di$}.
\end{equation}
Because $T_0$ is real linear, we obtain
$T_0(I_a(f')) = T_0(f) - T_0(f(a)\unit)$ on $\Di$.
Therefore, $\wh{\{T_0(I_a(f'))\}'} = \wh{T_0(f)'} - \wh{T_0(f(a)\unit)'}$ on $\M$.
Equality \eqref{thm2.1} shows that $\wh{T_0(f(a)\unit)'} = 0$ on $\sa$, and thus
$\wh{T_0(f(a)\unit)'} = 0$ on $\M$, since $\sa$ is a boundary for $\A$.
We deduce that
$\wh{\{T_0(I_a(f'))\}'} = \wh{T_0(f)'}$ on $\M$.

We show that $\W$ is surjective.
Let $\wh{v_0}\in\A$. By the surjectivity of $T_0 \colon \Ao \to \Ao$,
there exists $g \in \Ao$ such that $T_0(g) = I_a(v_0)$ on $\Di$,
and hence $\wh{T_0(g)'}=\wh{I_a(v_0)'}=\wh{v_0}$ on $\M$ by \eqref{thm2.3}.
Equality \eqref{thm2.5} shows that $\wh{\{T_0(I_a(g'))\}'} = \wh{T_0(g)'}=\wh{v_0}$ on $\M$.
By \eqref{thm2.4}, $\W(\wh{g'}) = \wh{\{T_0(I_a(g'))\}'} = \wh{v_0}$ on $\M$,
which shows the surjectivity of $\W \colon \A \to \A$.
Hence $\W$ is a surjective real linear isometry on the uniform algebra $(\A, \DV{\cdot}_\M)$.
By \cite[Theorem~1.1]{miu1} (see also \cite[Theorem~3.3]{hat}),
there exist a continuous function
$u_1 \colon \sa \to \T$, a closed and open set $G_0 \subset \sa$ and a homeomorphism
$\varrho \colon \sa \to \sa$ such that
\begin{equation}\label{thm2.8}
\W(\wh{v})(\eta) =
\begin{cases}
u_1(\eta)\wh{v}(\varrho(\eta)) & \eta \in G_0 \\[2mm]
u_1(\eta)\ov{\wh{v}(\varrho(\eta))} & \eta \in \sa \setminus G_0
\end{cases}
\end{equation}
for all $\wh{v} \in \A$. Then $\W(\wh{\unit}) = u_1$ on $\sa$.

By the surjectivity of $\W$, there exists $v_1 \in \H$ such that $\W(\wh{v_1}) = \wh{\unit}$.
Then $\V{\wh{v_1}}_\M = \V{\W(\wh{v_1})}_\M = 1$, since $\W$ is a real linear isometry.
Equality \eqref{thm2.8} implies that 
$\W(\wh{\unit})\W(\wh{v_1}^2) = u_1 \W(\wh{v_1}^2) = \W(\wh{v_1})^2$
on $\sa$, and thus $\W(\wh{\unit})\W(\wh{v_1}^2) = \wh{\unit}$ on $\sa$.
Since $\sa$ is a boundary for $\A$, we have $\W(\wh{\unit})\W(\wh{v_1}^2) = \wh{\unit}$
on $\M$.
Note that $|\W(\wh{v_1}^2)(\eta)| \leq 1$ for all $\eta \in \M$,
since $\V{\W(\wh{v_1}^2)}_\M = \V{\wh{v_1}^2}_\M = 1$.
Since $\V{\W(\wh{\unit})}_\M=\V{\wh{\unit}}_\M=1$, we obatin
$|\W(\wh{\unit})(\eta)| \leq 1$ for $\eta \in \M$.
Combining these two inequalities with the equality
$\W(\wh{\unit})(\eta)\W(\wh{v_1}^2)(\eta) = 1$ for $\eta \in \M$,
we deduce $|\W(\wh{\unit})| = 1$ on $\M$.
Since $\W(\wh{\unit})$ is the Gelfand transform of some analytic function on $\Di\subset\M$,
the maximum modulus principle shows that $\W(\wh{\unit})$ is a constant function
of modulus one. Hence, $\W(\wh{\unit}) = c_1$ on $\M$ for some $c_1 \in \T$.

Since $\W$ is bijective and real linear, so is the mapping $\ov{c_1}\W \colon \A \to \A$.
By \eqref{thm2.8}, we see that
$\ov{c_1}\W(\wh{v_1} \wh{v_2}) = \ov{c_1}\W(\wh{v_1})\ov{c_1}\W(\wh{v_2})$ on $\sa$
for all $\wh{v_1}, \wh{v_2} \in \A$.
Because $\sa$ is a boundary for $\A$, we see that $\ov{c_1}\W$ is multiplicative, and
hence it is a bijective real-algebra automorphism on $\A$.
By \cite[Themrem~2.1]{hat}, there exist a homeomorphism $\rho_1 \colon \M \to \M$
and a closed and open set $G_1$ of $\M$ such that 
$$
\ov{c_1}\W(\wh{v})(\eta) =
\begin{cases}
\wh{v}(\rho_1(\eta)) & \eta \in G_1 \\[2mm]
\ov{\wh{v}(\rho_1(\eta))} & \eta \in \M \setminus G_1
\end{cases}
$$
for all $\wh{v} \in \A$. By the connectedness of $\M$, we deduce that $G_1 = \M$
or $G_1 = \emptyset$, and hence there exists $\et \in \set{\pm 1}$ such that
\begin{equation}\label{thm2.9}
\W(\wh{v})(\eta) = c_1[\wh{v}(\rho_1(\eta))]^{\et}
\end{equation}
for all $\wh{v} \in \A$ and $\eta \in \M$.

Since $\W(\wh{\id}) \in \A$, there exists $h_1 \in \H$ such that $\W(\wh{\id}) = \wh{h_1}$.
Then $|\wh{h_1}|=|\wh{\id}\circ\rho_1|$ on $\M$ by \eqref{thm2.9}.
Thus $\V{\wh{h_1}}_\M = \V{\wh{\id} \circ \rho_1}_\M = 1$,
which implies $h_1(\Di) \subset \Db$.
Because $\rho_1$ is a homeomorphism on $\M$, we observe that $h_1$ is a
non-constant analytic function on $\Di$.
The open mapping theorem yields that
$h_1(\Di)$ is an open set, and hence $h_1(\Di) \subset \Di$.
This implies that $|\wh{h_1}|<1$ on $\Di$.
Since $|\wh{\id}| = 1$ on $\M \setminus \Di$ (see \cite[p. 161]{hof}),
the equality $|\wh{\id} \circ \rho_1| = |\wh{h_1}|$ shows that
$\rho_1(\Di) \subset \Di$.

The above arguments can be applied to the surjective real linear isometry $\W^{-1}$.
Then there exist a homeomorphism $\rho_{-1} \colon \M \to \M$ with $\rho_{-1}(\Di)\subset\Di$
and a constant
$s_3 \in \set{\pm 1}$
such that
\begin{equation}\label{thm2.10}
\W^{-1}(\wh{v})(\eta) = c_{-1}[\wh{v}(\rho_{-1}(\eta))]^{s_3}
\end{equation}
for all $\wh{v} \in \A$ and $\eta \in \M$. 
By \eqref{thm2.9} and \eqref{thm2.10}, 
the equality $\wh{v} = \W^{-1}(\W(\wh{v}))$ shows that
$$
\wh{v}(\eta)
= \W^{-1}(c_1 [\wh{v} \circ \rho_1]^{\et})(\eta)
= c_{-1} \Bigl[ c_1 [\wh{v} \circ \rho_1]^{\et} (\rho_{-1}(\eta)) \Bigr]^{s_3}
$$
for all $\wh{v} \in \A$ and $\eta \in \M$, where
$[\wh{v} \circ \rho_1]^{s_2}(\eta) = [\wh{v}(\rho_1(\eta))]^{s_2}$.
For $\wh{v}=\wh{\unit}$, the above equalities show that $1=c_{-1}[c_1]^{s_3}$,
and hence $\wh{v} = [\wh{v} \circ \rho_1 \circ \rho_{-1}]^{\et s_3}$ for all $\wh{v}\in\A$.
Substituting $\wh{v} = \wh{\id}$ into this equality to obtain
$\wh{\id} = [\wh{\id} \circ \rho_1 \circ \rho_{-1}]^{\et s_3}$ on $\M$.
Since $\rho_j(\Di) \subset \Di$ for $j=\pm 1$, we obtain
$z = [\rho_1(\rho_{-1}(z))]^{\et s_3}$ for all $z \in \Di$.
This shows that $\Di \subset \rho_1(\Di)$, and therefore $\rho_1(\Di)=\Di$.
Hence $\rho_1|_\Di \colon \Di \to \Di$ is a homeomorphism.
For each $z \in \Di$,
$c_1[\rho_1(z)]^{\et} = c_1[\wh{\id}(\rho_1(z))]^{\et} = \wh{h_1}(z) = h_1(z)$
by \eqref{thm2.9} with $\W(\wh{\id}) = \wh{h_1}$, and thus 
$[\rho_1|_\Di]^{\et} = \ov{c_1}h_1$ is analytic on $\Di$.
Since $\rho_1|_\Di \colon \Di \to \Di$ is a homeomorphism,
there exist $\la \in \T$ and $b \in \Di$ such that
\begin{equation}\label{thm2.11}
[\rho_1(z)]^{\et} = \la\, \fr{z-b}{1-\bar{b}z}
\qq (z \in \Di)
\end{equation}
(see, for example, \cite[Theorem~12.6]{rud}).

Let $f \in \Ao$. Equality \eqref{thm2.5}, combined with \eqref{thm2.4} and \eqref{thm2.9}, shows that
$$
\wh{T_0(f)'}(\eta)
= \wh{\{T_0(I_a(f'))\}'}(\eta) = \W(\wh{f'})(\eta)
= c_1 [\wh{f'}(\rho_1(\eta))]^{\et}
$$
for every $\eta \in \M$. In particular, since $\rho_1(\Di) = \Di$,
$$
T_0(f)'(z) 
= c_1 [f'(\rho_1(z))]^{\et}
\qq (z \in \Di).
$$
Equality \eqref{thm2.12}, applied to $T_0(f)$ instead of $f$, shows that
$I_a(T_0(f)') = T_0(f) - T_0(f)(a)\unit$ on $\Di$, that is,
$$
T_0(f)(z) = T_0(f)(a) + I_a(T_0(f)')(z)
$$
for all $z \in \Di$.
Recall that $T_0(f)(a) = c_0[f(a)]^{\ez}$ by \eqref{thm2.2}, and that
$[\rho_1(z)]^{\et} = \la(z-b)/(1-\bar{b}z)$ for $z \in \Di$ by \eqref{thm2.11}.
By the definition of $I_a$ with $T_0(f)' = c_1[f' \circ \rho_1]^{s_2}$, we get
\begin{align*}
T_0(f)(z)
&=
T_0(f)(a) + I_a(T_0(f)')(z)
=c_0[f(a)]^{\ez} + \int_{[a,z]} c_1 [f'(\rho_1(\zeta))]^{\et}\, d\zeta \\
&=
c_0[f(a)]^{\ez} + \int_{[a,z]} c_1 
\left[ f' \left( \left[ \la\, \fr{\zeta-b}{1-\bar{b}\zeta} \right]^{\et} \right) \right]^{\et}\, d\zeta
\end{align*}
for all $f \in \Ao$ and $z \in \Di$.
Since $T_0 = T-T(0)$ by \eqref{T_0}, we obtain the forms displayed in Theorem~\ref{thm2}.

Conversely, if $T$ is one of the above forms, then we can check that it is a surjective
isometry on $\Ao$. This completes the proof.
\end{proof}

On behalf of all authors, the corresponding author states that there is no conflict of interest.

\begin{ack}
The authors of the paper are deeply grateful to the referees for their valuable comments
and suggestions that significantly improve the former manuscript.
\end{ack}




\end{document}